\documentclass[11pt]{article}

\usepackage[margin=1in]{geometry}
\usepackage{microtype}
\usepackage{amsmath}
\usepackage{amssymb}
\usepackage{amsfonts}
\usepackage{amsopn}
\usepackage{amsthm}

\allowdisplaybreaks
\newcommand{\tr}{\mathrm{Tr}}

\usepackage{graphicx}
\usepackage{epstopdf}
\usepackage{ifpdf}
\ifpdf
  \DeclareGraphicsExtensions{.eps,.pdf,.png,.jpg}
\else
  \DeclareGraphicsExtensions{.eps}
\fi
\graphicspath{{figs/}}

\usepackage{enumerate}
\usepackage{algorithmic}
\usepackage{xcolor}
\colorlet{RED}{red}
\usepackage[normalem]{ulem} 
\usepackage[numbers,sort&compress]{natbib}
\usepackage[hidelinks]{hyperref}
\usepackage[capitalize,nameinlink]{cleveref}

\numberwithin{equation}{section}

\theoremstyle{plain}
\newtheorem{theorem}{Theorem}[section]
\newtheorem{lemma}[theorem]{Lemma}
\newtheorem{corollary}[theorem]{Corollary}
\newtheorem{proposition}[theorem]{Proposition}
\newtheorem{definition}[theorem]{Definition}

\theoremstyle{definition}
\newtheorem{assumption}[theorem]{Assumption}
\theoremstyle{remark}
\newtheorem{remark}[theorem]{Remark}

\crefname{assumption}{Assumption}{Assumptions}
\crefname{hypothesis}{Hypothesis}{Hypotheses}
\crefname{fact}{Fact}{Facts}

\newenvironment{keywords}{\par\medskip\noindent\textbf{Keywords.} \ignorespaces}{\par}
\newenvironment{MSCcodes}{\par\medskip\noindent\textbf{MSC codes.} \ignorespaces}{\par}
\newcommand{\email}[1]{\href{mailto:#1}{\texttt{#1}}}

\title{Continuous Policy and Value Iteration for Relaxed Stochastic Control Problems and Its Convergence}

\author{Qi Feng\thanks{Department of Mathematics, Florida State University, Tallahassee, FL 32306; \email{qfeng2@fsu.edu}.}
\and Gu Wang\thanks{Department of Mathematical Sciences, Worcester Polytechnic Institute, Worcester, MA 01609; \email{gwang2@wpi.edu}.}}

\date{October 1, 2026}

\hypersetup{
  pdftitle={Continuous Policy and Value Iteration for Relaxed Stochastic Control Problems and Its Convergence},
  pdfauthor={Q. Feng, and G. Wang}
}

\begin{document}

\maketitle
{\renewcommand{\thefootnote}{}\footnotetext{Qi Feng is partially supported by the National Science Foundation under grant \#DMS-2420029. Gu Wang is partially supported by National Science Foundation grant \#DMS - 2206282.}}

\begin{abstract}
We introduce a continuous policy-value iteration algorithm where the approximations of the value function of a stochastic control problem and the optimal control are simultaneously updated through Langevin-type dynamics. This paper focuses on the entropy-regularized relaxed control problem with infinite horizon. We establish policy improvement and demonstrate convergence to the optimal control under the monotonicity condition of the Hamiltonian. By utilizing Langevin-type stochastic differential equations for continuous updates along the policy iteration direction, our approach enables the use of distribution sampling techniques in machine learning to optimize the value function and identify the optimal control simultaneously. Numerical experiments are presented for LQ model with a non-concave reward, and a non-LQ example with a skewed-Laplace optimal control.
\end{abstract}

\begin{keywords}
Hamilton-Jacobi-Bellman equation; Reinforcement learning; Policy gradient ascent; Stochastic control; Relaxed control; Langevin dynamics.
\end{keywords}

\begin{MSCcodes}
93E20, 93E35, 60H10
\end{MSCcodes}

\section{Introduction}
The standard approaches for solving stochastic optimal control problems rely on the dynamic programming principle (i.e., Hamilton-Jacobi-Bellman (HJB) equations) \cite{bellman1954theory, bellman1966dynamic} or the maximum principle \cite{boltyanski1960maximum}, both of which involve highly non-linear partial differential equations for general reward/cost functions and state dynamics, and closed-form solutions are rarely obtained. These equations can either be solved numerically, using methods ranging from the classical finite difference approach \cite{BS1991} to deep learning methods \cite{HJE2018}, or their solutions can be approximated in closed form \cite{GW2020}, accompanied by error analysis based on duality arguments.

More recently, inspired by the (stochastic) gradient descent method for minimizing a given non-convex function, policy iteration algorithms for stochastic control problems have gained significant attention  \cite{puterman1979convergence, puterman1981convergence, bertsekas2015value, jacka2017policy, kerimkulov2020exponential}. In particular, in the relaxed control framework, where controls are randomized at every time and state \cite{FN1984,ZHou1992} and entropy regularization is incorporated into the cost functional, policy improvements are seen through the lens of reinforcement learning (RL) \cite{wang2020reinforcement}.
In the literature, RL \cite{bertsekas2012dynamic, szepesvari2022algorithms} and Q-learning \cite{bertsekas2011dynamic} are traditionally formulated in a discrete time setting. \cite{wang2020reinforcement} extends the discussion to a continuous-time, linear-quadratic setting with relaxed control framework. This advancement significantly enhances the applicability of RL to financial mathematics \cite{dai2023learning, denkert2025control, chau2024continuous} and stochastic control \cite{tran2024policy, tang2024regret,
possamai2024policy, dianetti2024exploratory, zhao2023policy, vsivska2024gradient}, and \cite{huang2022convergence} recently extends this approach to a broader class of control problems.

In most of the above literature, although state variables follow continuous-time dynamics, the policy iterations are still updated discretely along the iteration direction. Thus the mathematical argument for the convergence of the approximated value and policy to the optimal ones is highly delicate, once it is beyond the linear-quadratic case in \cite{wang2020reinforcement}, e.g., the approach adopted in \cite{huang2022convergence} can only deal with the case where the diffusion coefficients of the state dynamics are not controlled. Furthermore, at each step of the policy iteration, one has to evaluate the associated value function, either by solving the associated partial differential equation or approximating it by learning algorithms, posing significant challenges, particularly in high-dimensional problems. \cite{ma2024convergence2, ma2024convergence} recently proposed a new policy improvement algorithm that, while being discrete in nature, removes the necessity of solving the value function at each iteration, thereby enhancing both the efficiency and applicability of these methods. Gradient flow and actor-critic strategy are further studied in this direction  \cite{jia2022policy, zhou2024solving,  zhou2023policy}, while the actor-critic framework does not directly provide a Langevin dynamics based gradient descent algorithm for policy improvements as done in this paper.

To the best of our knowledge, we are the first to introduce iteration-continuous dynamics for both value and policy approximations, alongside continuous time dynamics for the state variable in the relaxed control framework, where both the drift and diffusion coefficients are controlled. For relaxed control problems, our algorithm establishes a connection to distribution sampling problems through continuous-time mean-field Langevin dynamics \cite{chen2022uniform, chizat2022mean, hu2021mean, bayraktar2024exponential}. This connection is the central theme of the present paper: the optimal randomized strategy is characterized as the Gibbs invariant distribution of a Langevin dynamics whose potential is the Hamiltonian evaluated at the current value iterate, so that solving the relaxed control problem becomes a distribution sampling problem in which the target distribution is learned simultaneously with the value function. By incorporating our policy-continuous dynamics, we extend the policy improvement algorithm to policy gradient ascent/descent methods, which has the potential to incorporate stochastic gradient ascent/descent algorithms \cite{welling2011bayesian, zhang2017hitting, raginsky2017non, Dalalyan2020} for high-dimensional problems.  Compared to the continuous dynamics with gradient flow in \cite{vsivska2024gradient,ZZH2025}, we only update the policy locally and infinitesimally at every step in the iteration, instead of its whole path at every point in time. Furthermore, we do not need to keep track of the dynamics of the state under the new strategy, nor do we use the stochastic maximum principle, thus avoiding the need to solve the forward and backward system of stochastic differential equations as in \cite{vsivska2024gradient}. We also point out that one of the key features of our continuous policy dynamics is its inhomogeneous nature - the objective function in the Langevin dynamics varies along the iteration. It leads to an interesting mathematical problem on the convergence of general inhomogeneous diffusion, which we leave for future studies.

 The paper is organized as follows. In Section \ref{sec: relax}, we define the relaxed control problem and its associated HJB equation. In Section \ref{sec: relax PGA}, we propose continuous iteration dynamics for both the value function and the control which informs the distribution strategy and show the policy improvement result. In Section \ref{sec: convergence}, we prove the exponential convergence rate of the proposed continuous dynamics under the monotonicity condition. An Linear-Quadratic (LQ) example satisfying the monotonicity condition is provided in Section \ref{sec: example}, together with numerical experiments for LQ model with a non-concave reward, and a non-LQ example with a skewed-Laplace optimal control.

\section{Relaxed Stochastic Control Problem}\label{sec: relax}
Consider an infinite horizon relaxed control problem, where at every time $t\geq 0$, the agent randomizes the strategy  $u\in\mathbb R^n$ according to a distribution density $\pi_{t}$, which controls both the drift and the diffusion of the state variable $X\in \mathbb R^d$, so that it follows the exploratory dynamics
\begin{equation}
    dX^\pi_t = \tilde b(X^\pi_t,\pi_t)\,dt + \tilde\sigma(X^\pi_t,\pi_t)\,dW_t, \quad X_0 = x,\label{eq: X-dyn}
\end{equation}
where for any distribution density $\pi$ in $\mathbb R^n$, $\tilde b (x,\pi) = \int_{\mathbb R^n} b(x,u)\pi(u)du$, $\tilde \sigma(x,\pi) = \left(\int_{\mathbb R^n} \sigma\sigma^{\top}(x,u)\pi(u)du\right)^{1/2}$, with $b: \mathbb R^d\times \mathbb R^n \rightarrow \mathbb R^d$, $\sigma: \mathbb R^d\times \mathbb R^n \rightarrow \mathbb R^{d\times k}$ satisfying Assumption \ref{assumption:parameter} below, and $W$ is a $d$-dimensional Brownian motion. The admissible set $\mathcal A(x)$ of the relaxed control $\pi = \{\pi_t\}_{t\geq 0}$ is defined as follows.
\begin{definition}\label{def:admissible}
For every initial state $x\in\mathbb R^d$, the admissible control set $\mathcal A(x)$ is the collection of $\pi = \{\pi_t\}_{t\geq 0}$ such that: 
(i) $\pi_t \in \mathcal P(\mathbb R^n)$ for each $t\geq 0$, where $\mathcal{P}(\mathbb R^n)$ is the space of absolutely continuous probability distributions on $\mathbb R^n$ (identified with their densities). (ii) $\mathbb E[\int_0^T (|\tilde b(X^\pi_t,\pi_t)|^2 + \tr[\tilde\sigma\tilde \sigma^\top(X^\pi_t,\pi_t)]) dt ]<\infty$ for each $T > 0$.
(iii) $J(x,\pi) < \infty$.
\end{definition}
The agent's goal is to maximize the expected running payoff $f:\mathbb R^d\times \mathbb R^n \rightarrow \mathbb R$, regularized by entropy $G(\pi_t) = \int_{\mathbb R^n}\pi_t(u)\ln\pi_t(u)\,du$ for $t\geq 0$, i.e., $V(x) =  \sup_{\pi\in\mathcal A(x)} J(x,\pi)$, where
\begin{equation}\label{definition:J}
	J(x,\pi)=\mathbb E\!\Big[\int_0^\infty e^{-\beta t}\!\left(\tilde f(X^\pi_t,\pi_t)
-\lambda G(\pi_t)\right)dt\Big],
\end{equation}
with $\tilde f(x,\pi) = \int_{\mathbb R^n} f(x,u)\pi(u)du$ the exploratory reward for any $\pi\in \mathcal P(\mathbb R^n)$ (cf. \cite{wang2020reinforcement} for the motivation of this model and its connection with reinforcement learning). 
We make the following assumptions on model parameters.
\begin{assumption}\label{assumption:parameter}
(i) $f$, $b$ and $\sigma\sigma^\top$ are bounded and $C^2$ in both $x$ and $u$, and all the first-order and second-order derivatives are uniformly bounded; (ii) $\sigma\sigma^\top$ is uniformly positive definite, i.e.\ $\sigma\sigma^\top(x,u)\succeq c\,I$ for some constant $c>0$; (iii) $V(x) < \infty$.
\end{assumption}
Due to the Markovian and infinite horizon setting, the value function is time-independent and is expected to solve the corresponding HJB equation
\begin{align}
 -\beta v(x)  +\sup_{\pi\in\mathcal P(\mathbb R^n)} \int_{\mathbb R^n} \left(\mathbf{H}(x,u,v_x,v_{xx}) -\lambda \ln\pi(u)\right)\pi(u)du =0,\nonumber
\end{align}
where $\mathbf{H}(x,u,v_x,v_{xx}) = f(x,u) + b^\top(x,u)v_x(x)+ \frac{1}{2}\tr[\sigma\sigma^\top(x,u):v_{xx}(x)]$, and $v_x$, $v_{xx}$ are the gradient and Hessian of $v$ with respect to $x$, respectively. The first-order condition gives
\begin{equation}
	\hat\pi(x,u) = \frac{\exp\left(\frac{1}{\lambda}\mathbf{H}(x,u,v_x,v_{xx})\right)}{\int_{\mathbb R^n} \exp\left(\frac{1}{\lambda}\mathbf{H}(x,u',v_x,v_{xx})\right)du'}. \nonumber
\end{equation}
If the optimal value function $\hat v = V$ is known, then the optimal distribution $\hat\pi$ (by plugging $\hat v$ in the above first order condition) is the density of the invariant distribution of (for a given $x$ and $n$-dimensional Brownian motion $B$ independent of $W$) the following diffusion,
\begin{equation}\label{eq:u-dynamics}
	du^x_\tau = \nabla_u \mathbf{H} (x,u^x_\tau,\hat v_x,\hat v_{xx})d\tau+ \sqrt{2\lambda}dB_\tau,
\end{equation}
provided that $\int_{\mathbb R^n} \exp\left(\frac{1}{\lambda}\mathbf{H}(x,u,\hat v_x,\hat v_{xx})\right)du <\infty$ \citep{CHS1987, roberts1996exponential}, and $\hat\pi$
solves the stationary Fokker--Planck equation $\mathcal{L}^{*} \hat \pi = 0$, where $\mathcal{L}^{*}$ is the formal $L^2$-adjoint of the infinitesimal generator $\mathcal{L}$ for \eqref{eq:u-dynamics}.
However, for general model parameters, the HJB equation does not have a closed-form solution and the optimal value function $\hat v$ is not known \emph{a priori}. In this paper, instead of solving the HJB equation corresponding to the optimal solution directly \citep{TZZ2022}, we propose a method that combines reinforcement learning and stochastic gradient descent from deep learning techniques, in which the optimal policy and value function are learned simultaneously through coupled continuous-time dynamics. It avoids evaluating the value function corresponding to the updated policy at each step of the iteration, e.g., by solving the associated PDE or approximating it with a neural network, as proposed in \cite{wang2020reinforcement}, where the iteration is discrete in nature.

\section{Policy and Value Iterations}\label{sec: relax PGA}
Given $\pi^x_0 \in \mathcal P(\mathbb R^n)$ for each $x\in \mathbb R^d$, consider $(v^0,\rho^0)$ as the starting point of the iteration algorithm, where $\rho^0 = \{\rho^0_t = \pi_0, 0\leq t<\infty\}$ and $v^0(x) = J(x,\rho^0)$. The policy iteration follows 
\begin{equation}
	du^x_\tau =\nabla_u \mathbf{H}(x,u^x_\tau,v^\tau_x(0,x),v^\tau_{xx}(0,x))d\tau + \sqrt{2\lambda}dB_\tau, \quad u_0 \sim \pi_0,\label{eq:DynU}
\end{equation}
where $v^\tau$ is the result of the value iteration introduced later, and is motivated by the above observation that the optimal randomized strategy is the invariant distribution of \eqref{eq:u-dynamics}. Notice that if $v_x$ and $v_{xx}$ are not updated throughout the policy iteration, then given $v$, as $\tau\rightarrow \infty$, the distribution of $u_\tau$ converges to $\pi_{\infty} = \frac{\exp\left(\frac{1}{\lambda}\mathbf{H}(x,u,v_x,v_{xx})\right)}{\int_{\mathbb R^n} \exp\left(\frac{1}{\lambda}\mathbf{H}(x,u',v_x,v_{xx})\right)du'}$, which matches the policy iteration method proposed in \cite{wang2020reinforcement,huang2022convergence}. Thus, in our framework, their method is a special case with only updating the value function at $\tau = \infty$ and starting another policy iteration given the new value function, without updating it until the policy becomes stationary, which explains its ``discrete'' nature. Our method updates the value function much more frequently, simultaneously with the policy.
For each $\tau$, $v^\tau = J(x,\rho^\tau)$ for a potentially time-inhomogeneous strategy $\rho^\tau$, and thus may be a function of time (we will show that the strategy will become time-homogeneous as $\tau\rightarrow \infty$, as expected in the optimal case). We update the value function by updating the corresponding strategy $\rho^{\tau}$, using the distribution of $u_s$ for $0\leq s\leq \tau$. Notice that starting from $\pi^x_0\in \mathcal P(\mathbb R^n)$, the density function of $u^x_\tau$ depend on $x$, and thus should be denoted as $\pi^x_\tau$. In the following, since we will evaluate this density function at different state, we denote it as $\pi_\tau(x)$, to emphasize this dependence, while omitting of the argument $u$ for the ease of notation, unless ambiguity arises, for example, in an integral with respect to the density function.
\begin{definition}\label{def: policy}
For each $\tau\geq 0$, given $\{\pi_s(x),0\leq s\leq \tau\}$ (the distribution of $u_s$ for $0\leq s\leq \tau$) for every $x \in \mathbb R^d$, define $\boldsymbol\rho^\tau = \{\rho^\tau_t(X^\tau_t)\}_{0\leq t<\infty}$ as
\begin{equation}
\rho^\tau_t(x) = \begin{cases}
\pi_{\tau-t}(x), &0\leq t \leq \tau;\\
\pi_0(x), & t > \tau.
\end{cases}
\end{equation}
The corresponding state process $X^{\tau}$ satisfies
\begin{equation}
	dX^\tau_t = \tilde b(X^\tau_t,\rho^\tau_t(X^\tau_t))dt +\tilde\sigma(X^\tau_t,\rho^\tau_t(X^\tau_t))dW_t,\nonumber
\end{equation}
and the corresponding value function is
\begin{align}
v^{\tau}(t,x)
= \mathbb E\Big[\int_t^\infty e^{-\beta(s-t)}\Big(\tilde f(X^\tau_s,\rho^\tau_s(X^\tau_s)) - \lambda G(\rho^\tau_s(X^\tau_s))\Big)ds\,\Big|\,X^\tau_t = x\Big],\label{definition:v-tau}
\end{align}
i.e., $v^\tau(0,x) = J(x,\rho^\tau)$.
\end{definition}
By the above definition, as $\tau$ increases and new distributions $\pi_\tau$ are generated from the policy iteration \eqref{eq:u-dynamics}, we update the strategy used in generating the value function $v^\tau$ by using the newest $\pi_\tau$ at time $t=0$, and pushing the previous sequence of strategies (as a function of time) into the future. As the iteration proceeds, the initial suboptimal strategies are pushed further into the future and because of the positive discounting factor, become less relevant. We show in Proposition \ref{prop:policy-improvement} below that $\boldsymbol\rho^\tau$ becomes better as $\tau$ increases, in the sense that the policy improvement property holds. Finally, Propositions \ref{prop:convergence} and \ref{prop:optimality} show that $\pi_\tau$ converges to the optimal strategy, and $v^\tau$ also becomes optimal. Furthermore, at every $\tau$, we only make infinitesimal changes to $\boldsymbol\rho^\tau$, so the change in the corresponding value $v^\tau$ are infinitesimal, which allows us to characterize $v^\tau$ and its derivatives by simple continuous dynamics in Lemma \ref{dynamic v relax}.

Notice that in the above policy iteration, we only use $v^\tau(0,x)$ and its derivatives in $x$, thus in the following discussion, we focus on the continuous dynamics which update $v^\tau(0,x)$ along the $\tau$-direction instead of the $t$-direction. In particular, we drop the time argument ``$0$'' for the ease of notation (i.e., $v^\tau(x)$ stands for $v^\tau(0,x)$) unless $v^\tau(t,x)$ for general $t$ is discussed. Similarly, while the distribution of $u^x_\tau$ should be denoted as $\pi_\tau(x;u)$, the index $x$ and argument $u$ are dropped if no ambiguity arises. Also in the rest of the paper, $\tau$ always denotes the direction of algorithm iteration, i.e., we are updating the distribution of the control by letting $u$ evolve according to a diffusion, in which $x$ (the state) is fixed, and in order to update the policy for all states, we are supposed to run this diffusion for every $x$. Finally, for ease of notation, we follow the convention that for any process $z$, $z_t = z_0$ for $t<0$.

For a model with $n=d=1$ and $f$ independent of state $x$, the control problem is reduced to maximize $f(u)\pi(u) - \lambda\pi(u)\ln\pi(u)$, and the optimal solution is $\hat\pi = \frac{e^{f(u)/\lambda}}{\int_{\mathbb R} e^{f(u')/\lambda}du'}$ and the optimal value function is independent of $x$. The corresponding value
$
	v^\pi(x) =  \int_0^\infty e^{-\beta t} \int_{\mathbb R}\left(f(u)\pi(u) - \lambda \pi(u)\ln\pi(u)\right)du\,dt, \nonumber
$
is independent of $x$. If policy $u_\tau$ follows the state-independent iteration $
	du_\tau =  \nabla f(u_\tau)d\tau + \sqrt{2\lambda}dB_\tau$, 
the invariant distribution of $u_\tau$ equals the optimal strategy.
For more general cases, state-specific policy iterations are needed. We denote the Hamiltonian evaluated at the round $\tau$ of the value iteration as
\begin{align}\label{H_tau}
\mathbf{H}^\tau(x,u):=\mathbf{H}(x,u,v^\tau_x,v^\tau_{xx}),
\end{align}
and we make the following assumptions:
\begin{assumption}\label{assumption-T}
	$\pi_0$, the initial distribution density of $u$, is continuous in $x$. For each $\tau \geq 0$ and sufficiently small $h>0$: 
(i) $\mathbf\rho^\tau \in\mathcal A(x)$ and $v^\tau \in C^{1,4}$; (ii)  $\mathbb E\left[\int_0^{\tau +h} \| v^{\tau}_x(s,X^{\tau+h}_s)\|^2ds\right]< \infty$;
     (iii) $\mathbb E\Big[\|v^{\tau-s+l}_x(X^{\tau+h}_s)\|^2$ $+ \|v^{\tau-s+l}_{xx}(X^{\tau+h}_s)\|^2$\newline $+\|\nabla_u\ln \pi_{\tau-s+l}(X^{\tau+h}_s,u_{\tau-s+l})\|^2\Big]$ is bounded in $0\leq l \leq h$, for every $0\leq s\leq \tau+h$; 
     (iv) $\mathbb E\big[\big|\int_{\mathbb R^n} \left(\mathbf{H}(\cdot,u,v^\tau_x,v^\tau_{xx}) - \lambda\ln\pi_{\tau+h-s}\right)\pi_{\tau+h-s}(x,u)du$ $-\beta v^{\tau}\big|(X^{\tau+h}_s)\big]$ is bouned in $0\leq s\leq h$;
    (v) Given $x$, $\int_{\mathbb R^n} \pi_\tau(x,u)\ln\pi_\tau(x,u)du$ is uniformly bounded in $\tau$.
\end{assumption}
We first show the following policy improvement result.
\begin{proposition}\label{prop:policy-improvement}
For $0<\tau_1 < \tau_2$, $v^{\tau_1}(t,x)\leq v^{\tau_2}(t,x)$.
\end{proposition}
\begin{proof}
We discuss the case of $t =0$, while the following argument works for any $t\geq 0$. By It\^o's formula, $v^\tau$ satisfies (for $h>0$)
\begin{align}
	&e^{-\beta t}v^{\tau}(t,X^{\tau+h}_t) - v^{\tau}(0,x)\nonumber\\
	=& \int_0^{t} e^{-\beta s} \Big(v^{\tau}_t -\beta v^{\tau} + \tilde b^\top(X^{\tau+h}_s,\pi_{\tau+h-s}(X^{\tau+h}_s))v^{\tau}_x(s,X^{\tau+h}_s)\Big.\nonumber\\
    &\Big.+ \frac{1}{2}\tr[\tilde\sigma\tilde\sigma^\top(X^{\tau+h}_s,\pi_{\tau+h-s}(X^{\tau+h}_s)):v^{\tau}_{xx}(s,X^{\tau+h}_s)]\Big)ds\nonumber\\
    &+ \int_0^{t} e^{-\beta s} \tilde \sigma^\top(X^{\tau+h}_s,\pi_{\tau+h-s}(X^{\tau+h}_s))v^{\tau}_x(s,X^{\tau+h}_s)dW_s.\nonumber
\end{align}
Since $v^{\tau}_x(s,X^{\tau+h}_s)$ is square integrable for sufficiently small $h$ by Assumption \ref{assumption-T} and $\tilde \sigma$ is bounded, the integral with respect to $W$ is a martingale. Letting $t=\tau+h$ and taking expectation on both sides of the above equation (in the following discussion, we summarize the arguments of all the functions inside an integral at the end of the integrand for the ease of notation),
\begin{align}
	&v^{\tau}(0,x)
	= \mathbb E\Big[\int_0^{\tau+h} e^{-\beta s} \Big(-v^{\tau}_t +\beta v^{\tau} - \tilde b^\top(\cdot,\pi_{\tau+h-s})v^{\tau}_x\Big.\Big.\nonumber\\
    &\Big.\Big. - \frac{1}{2}\tr[\tilde\sigma\tilde\sigma^\top(\cdot,\pi_{\tau+h-s}):v^{\tau}_{xx}]\Big)(s,X^{\tau+h}_s)ds+e^{-\beta(\tau+h)}v^{\tau}(\tau+h,X^{\tau+h}_{\tau+h})\Big].\label{eq:v-tau-zero-1}
\end{align}
On the other hand, from the definition of $v^{\tau}$ in Definition \ref{definition:v-tau},
\begin{align}
v^{\tau}(t,X^\tau_t) =& \mathbb E\Big[\int_t^{t+h} e^{-\beta(s-t)}\Big(\tilde f(\cdot,\pi_{\tau-s}) - \lambda\int_{\mathbb R^n}\pi_{\tau-s}\ln\pi_{\tau-s}(\cdot,u)du\Big)(X^\tau_s)ds\Big.\nonumber\\
+& \Big. e^{-\beta h}v^\tau(t+h,X^\tau_{t+h})\,\Big|\,\mathcal F^W_t\Big],
\end{align}
where $\{\mathcal F^W_t\}_{t\geq 0}$ is the filtration generated by $W$. 
Based on the smoothness assumption on $v^\tau$ in Assumption \ref{assumption-T}, by standard arguments, $v^{\tau}(t,x)$ satisfies,
\begin{align}
&v^{\tau}_t(t,x) -\beta v^{\tau}(t,x) +\tilde b^\top(\cdot,\pi_{\tau-t})v^{\tau}_x(t,x) + \frac{1}{2}\tr[\tilde\sigma\tilde\sigma^\top(\cdot,\pi_{\tau-t}):v^{\tau}_{xx}(t,x)] \nonumber\\
&+ \int_{\mathbb R^n}\left(f- \lambda \ln\pi_{\tau-t}\right)\pi_{\tau-t}(x,u)du = 0.\label{v-pde}
\end{align}
Substituting this into \eqref{eq:v-tau-zero-1},
\begin{align} &v^{\tau}(0,x) = \mathbb E\Big[e^{-\beta(\tau +h)}v^{\tau}(\tau +h,X^{\tau+h}_{\tau+h})\Big]\nonumber\\ 
&+ \mathbb E\Big[\int_0^{\tau +h} e^{-\beta s} \left(\tilde b(\cdot,\pi_{\tau-s}) - \tilde b(\cdot,\pi_{\tau+h-s})\right)^\top v^{\tau}_x(s,X^{\tau+h}_s)ds\Big]\nonumber\\ 
&+\frac{1}{2}\mathbb E\Big[\int_0^{\tau +h} e^{-\beta s} \tr[\left(\tilde \sigma\tilde\sigma^\top(\cdot,\pi_{\tau-s}) - \tilde \sigma\tilde\sigma^\top(\cdot,\pi_{\tau+h-s})\right):v^{\tau}_{xx}(s,X^{\tau+h}_s)]ds\Big]\nonumber\\ 
&+\mathbb E\Big[\int_0^{\tau +h}e^{-\beta s}\Big( \tilde f(\pi_{\tau-s}) - \lambda G(\pi_{\tau-s})\Big)(X^{\tau+h}_s)ds\Big]. \label{eq:v-tau-zero-2}
\end{align}
 Notice that $\rho^\tau_t = \pi_0$ for $t\geq \tau+h$, thus starting from $t=\tau+h$ and initial state $x$, adopting strategy $\rho^\tau$, the agent obtains the value
 \begin{align*}
 v^\tau(\tau+h,x) =& \mathbb E\Big[\int_{\tau+h}^\infty e^{-\beta (t-\tau-h)}(\tilde f(\rho^\tau_t)- \lambda G(\rho^\tau_t))(X^{\tau}_t)dt\Big|\,X^\tau_{\tau+h} = x\Big]\\
 =&\mathbb E\Big[\int_{\tau+h}^\infty e^{-\beta (t-\tau-h)} (\tilde f(\pi_0)- \lambda G(\pi_0))(X^{\tau}_t)dt\Big|\,X^\tau_{\tau+h}=x\Big]\\
 =& \mathbb E\Big[\int_{0}^\infty e^{-\beta t}(\tilde f(\pi_0)- \lambda G(\pi_0))(X^{0}_t)dt\Big|\,X^0_0 = x\Big]=v^0(0,x).
 \end{align*}
 A similar argument shows that $v^0(0,x) = v^{\tau+h}(\tau+h,x)$. Thus $v^{\tau}(\tau +h,X^{\tau+h}_{\tau+h}) = v^{\tau+h}(\tau +h,X^{\tau+h}_{\tau+h})$. Plugging this into \eqref{eq:v-tau-zero-2}, also using the fact that
 \begin{align*}
 &\mathbb E\Big[e^{-\beta(\tau +h)}v^{\tau+h}(\tau +h,X^{\tau+h}_{\tau+h})\Big] \nonumber\\
 =& v^{\tau+h}(0,x) -\mathbb E\Big[\int_0^{\tau+h} e^{-\beta s}\left(\tilde f(\pi_{\tau+h-s}) - \lambda G(\pi_{\tau+h-s})\right)(X^{\tau+h}_s)ds \Big],
 \end{align*}
the following holds
    \begin{align*}
	&v^\tau(0,x)=v^{\tau+h}(0,x)+\mathbb E\Big[\int_0^{\tau +h}e^{-\beta s}\left(\tilde f\left(\cdot,\pi_{\tau-s}\right)-\tilde f\left(\cdot,\pi_{\tau+h-s}\right)\right)(X^{\tau+h}_s)ds\Big]\nonumber\\
	&+\mathbb E\Big[\int_0^{\tau +h} e^{-\beta s} \left(\tilde b(\cdot,\pi_{\tau-s}) - \tilde b(\cdot,\pi_{\tau+h-s})\right)^\top v^{\tau-s}_x(X^{\tau+h}_s) ds\Big]\nonumber\\
     &+\frac{1}{2}\mathbb E\Big[\int_0^{\tau +h} e^{-\beta s} \tr[\left(\tilde \sigma\tilde\sigma^\top(\cdot,\pi_{\tau-s}) - \tilde \sigma\tilde\sigma^\top(\cdot,\pi_{\tau+h-s})\right):v^{\tau-s}_{xx}(X^{\tau+h}_s)]ds\Big]\nonumber\\
	&+ \mathbb E\Big[\int_0^{\tau +h}e^{-\beta s}\lambda\left(G(\pi_{\tau+h-s}) - G(\pi_{\tau-s})\right)(X^{\tau+h}_s)ds\Big],
\end{align*}
where we also used fact that $v^{\tau}(s,x) = v^{\tau-s}(0,x) = v^{\tau-s}(x)$ by our convention.

Since $\pi_{\tau}$ is the distribution density function of $u_\tau$, the above can be written as (with a slight abuse of notation where the expectation is taken in a larger probability space which supports the two independent Brownian motions $B$ and $W$, and and with deterministic coefficients in \eqref{eq:DynU}, $u$ is independent of $W$)
\begin{align}\label{eq:v-difference}
	&v^{\tau}(0,x) - v^{\tau+h}(0,x)\nonumber\\
	=&\mathbb E\Big[\int_0^{\tau +h}e^{-\beta s}\mathbb E\left[f(X^{\tau+h}_s,u_{\tau-s})-f(X^{\tau+h}_s,u_{\tau+h-s})\,\big|\,\mathcal F^W_s\right]ds\Big]\nonumber\\
	+&\mathbb E\Big[\int_0^{\tau +h} e^{-\beta s} \mathbb E\left[\left(b(\cdot,u_{\tau-s}) -  b(\cdot,u_{\tau+h-s})\right)^\top v^{\tau-s}_x(X^{\tau+h}_s)\,\big|\,\mathcal F^W_s\right] ds\Big]\nonumber\\
     +&\mathbb E\Big[\int_0^{\tau +h} \frac{e^{-\beta s}}{2} \mathbb E\Big[\tr[\left(\sigma\sigma^\top(\cdot,u_{\tau-s}) -  \sigma\sigma^\top(\cdot,u_{\tau+h-s})\right):v^{\tau-s}_{xx}(X^{\tau+h}_s)]\,\big|\,\mathcal F^W_s\Big]ds\Big]\nonumber\\
	+&\lambda \mathbb E\Big[\int_0^{\tau +h}e^{-\beta s}\mathbb E\Big[\ln\pi_{\tau+h-s}(X^{\tau+h}_s,u_{\tau+h-s})- \ln\pi_{\tau-s}(X^{\tau+h}_s,u_{\tau-s})\,\big|\,\mathcal F^W_s\Big] ds\Big].
\end{align}
With $\mathbf{H}^\tau$ defined in \eqref{H_tau} and denoting by $\nabla^2_{uu}$ the Jacobi matrix, applying It\^o's formula to $f$, $b$ and $\sigma\sigma^\top$ as functions of $u_\tau$ following \eqref{eq:DynU}, for each $0\leq s\leq \tau+h$,
\begin{align}
&\mathbb E\Big[f(X^{\tau+h}_s,u_{\tau-s})-f(X^{\tau+h}_s,u_{\tau+h-s})\,\big|\,\mathcal F^W_s\Big]	\nonumber\\
=&\mathbb E\Big[- \int_0^{h}\left((\nabla_u f)^\top\nabla_u \mathbf{H}^{\tau-s+l} + \lambda\tr[\nabla^2_{uu}f]\right)(X^{\tau+h}_s,u_{\tau-s+l})dl\,\Big|\,\mathcal F^W_s\Big]\nonumber\\
 &\mathbb E\left[\left(b(\cdot,u_{\tau-s}) -  b(\cdot,u_{\tau+h-s})\right)^\top v^{\tau-s}_{x}(X^{\tau+h}_s)\,\big|\,\mathcal F^W_s\right]\nonumber\\
 =& \mathbb E\Big[-\int_0^{h}\Big((\nabla_u b\,\nabla_u \mathbf{H}^{\tau-s+l})^\top v^{\tau-s}_{x} \Big.\Big.\nonumber\\
\qquad\Big.\Big.& + \lambda\sum\limits_{i=1}^d\tr[\nabla^2_{uu}b_i]v^{\tau-s}_{x_i}\Big)(X^{\tau+h}_s,u_{\tau-s+l})dl\,\Big|\,\mathcal F^W_s\Big]\nonumber\\
 &\mathbb E\Big[\tr[\Big(\sigma\sigma^\top(\cdot,u_{\tau-s}) -  \sigma\sigma^\top(\cdot,u_{\tau+h-s})\Big):v^{\tau-s}_{xx}(X^{\tau+h}_s)]\,\big|\,\mathcal F^W_s\Big]\nonumber\\
 =& \mathbb E\Big[-\int_0^{h} \Big(\nabla_u \tr[\sigma\sigma^\top:v^{\tau-s}_{xx}]^\top\nabla_u \mathbf{H}^{\tau-s+l} \Big.\Big.\nonumber\\
\qquad\Big.\Big.&+ \lambda\nabla^2_{uu}\tr[\sigma\sigma^\top:v^{\tau-s}_{xx}]\Big)(X^{\tau+h}_s,u_{\tau-s+l})dl\,\Big|\,\mathcal F^W_s\Big]\nonumber,
 \end{align}
which follow from the fact that the integral with respect to Brownian motion $B$ is a martingale by Assumption \ref{assumption:parameter}. Furthermore,
\begin{align*}
 &\lambda\mathbb E\Big[\ln\pi_{\tau+h-s}(X^{\tau+h}_s,u_{\tau+h-s}) - \ln\pi_{\tau-s}(X^{\tau+h}_s,u_{\tau-s})\,\big|\,\mathcal F^W_s\Big]\nonumber\\
 =&\lambda \mathbb E\Big[\int_0^{h}\Big(\frac{\nabla_\tau \pi_{\tau-s+l}}{\pi_{\tau-s+l}} + \frac{(\nabla_u \pi_{\tau-s+l})^\top\nabla_u \mathbf{H}^{\tau-s+l}}{\pi_{\tau-s+l}}\Big)(X^{\tau+h}_s,u_{\tau-s+l})dl\,\Big|\,\mathcal F^{W}_s\Big]\nonumber\\
 +&\lambda^2 \mathbb E\Big[\int_0^{h}\Big(\frac{\tr[\nabla^2_{uu} \pi_{\tau-s+l}]}{\pi_{\tau-s+l}}- \frac{(\nabla_u \pi_{\tau-s+l})^\top\nabla_u \pi_{\tau-s+l}}{\pi^2_{\tau-s+l}}\Big)\Big.\\
 &\qquad\left.(X^{\tau+h}_s,u_{\tau-s+l})dl\,\Big|\,\mathcal F^{W}_s\right]\\
 =& \lambda\mathbb E\Big[\int_0^{h} \Big(-\tr[\nabla^2_{uu} \mathbf{H}^{\tau-s+l}] + 2\lambda\frac{\tr[\nabla^2_{uu} \pi_{\tau-s+l}]}{\pi_{\tau-s+l}}\Big.\Big.\\
 &\qquad\Big.\Big. - \lambda\|\nabla_u\ln \pi_{\tau-s+l}\|^2\Big)(X^{\tau+h}_s,u_{\tau-s+l})dl\,\Big|\,\mathcal F^{W}_s\Big]\\
=&\lambda\mathbb E\Big[\int_0^{h} \Big(-\tr[\nabla^2_{uu} \mathbf{H}^{\tau-s+l}]- \lambda\|\nabla_u\ln \pi_{\tau-s+l}\|^2\Big)(X^{\tau+h}_s,u_{\tau-s+l})dl\,\Big|\,\mathcal F^{W}_s\Big],
\end{align*}
where the second equation follows from the Fokker-Planck equation for $\pi_{\tau-s+l}$, and the last equation holds because
\begin{equation}
	\mathbb E\Big[\frac{\tr[\nabla^2_{uu} \pi_{\tau-s+l}]}{\pi_{\tau-s+l}}(X^{\tau+h}_s,u_{\tau-s+l})\Big|\mathcal F^{W}_s\Big] =\int_{\mathbb R^n}\tr[\nabla^2_{uu} \pi_{\tau-s+l}](X^{\tau+h}_s,u)du = 0.\nonumber
\end{equation}
Plugging these calculations back into \eqref{eq:v-difference},
\begin{align*}
	&v^{\tau}(0,x) - v^{\tau+h}(0,x)\nonumber\\
	=&-\mathbb E\Big[\int_0^{\tau +h}e^{-\beta s}\int_0^{h} \Big((\nabla_u f)^\top\nabla_u \mathbf{H}^{\tau-s+l} + \lambda\tr[\nabla^2_{uu}f]\Big)(X^{\tau+h}_s,u_{\tau-s+l})dl\,ds\Big]\nonumber\\
	-&\mathbb E\Big[\int_0^{\tau +h} e^{-\beta s} \int_0^{h}\Big((\nabla_u b^\top v^{\tau-s}_x)^\top\nabla_u \mathbf{H}^{\tau-s+l}\Big.\Big.\nonumber\\
	&\qquad\Big.\Big.+ \lambda\sum\limits_{i=1}^d\tr[\nabla^2_{uu}b_i]v^{\tau-s}_{x_i}\Big)(X^{\tau+h}_s,u_{\tau-s+l})dl\, ds\Big]\nonumber\\
    -&\frac{1}{2}\mathbb E\Big[\int_0^{\tau +h} e^{-\beta s} \int_0^{h}\Big(\nabla_u \tr[\sigma\sigma^\top\, :v^{\tau-s}_{xx}]^\top\nabla_u \mathbf{H}^{\tau-s+l}\Big.\Big.\nonumber\\
	&\qquad\Big.\Big. + \lambda\nabla^2_{uu}\tr[\sigma\sigma^\top\, : v^{\tau-s}_{xx}]\Big)(X^{\tau+h}_s,u_{\tau-s+l})dl\, ds\Big]\nonumber\\
	+&\lambda \mathbb E\Big[\int_0^{\tau +h}e^{-\beta s}\int_0^{h}\Big(-\tr[\nabla^2_{uu} \mathbf{H}^{\tau-s+l}]\Big.\Big.\nonumber\\
	&\qquad\Big.\Big. - \lambda\|\nabla_u\ln \pi_{\tau-s+l}\|^2 \Big)(X^{\tau+h}_s,u_{\tau-s+l})dl\, ds\Big]\\
    =& -\mathbb E\Big[\int_0^{\tau+h} e^{-\beta s}\int_0^{h}(\nabla_u\mathbf{H}^{\tau-s})^\top\nabla_u \mathbf{H}^{\tau-s+l}(X^{\tau+h}_s,u_{\tau-s+l})dl\,ds \Big]\nonumber\\
    -&\lambda\mathbb E\Big[\int_0^{\tau+h} e^{-\beta s}\int_0^{h}\Big(\tr[\nabla^2_{uu} \mathbf{H}^{\tau-s+l}]+\tr[\nabla^2_{uu} \mathbf{H}^{\tau-s}]\Big)(X^{\tau+h}_s,u_{\tau-s+l})dl\,ds \Big]\nonumber\\
	-&\mathbb E\Big[\int_0^{\tau+h} e^{-\beta s}\int_0^{h}  \lambda^2\|\nabla_u\ln \pi_{\tau-s+l}\|^2 (X^{\tau+h}_s,u_{\tau-s+l})dl\,ds \Big].
\end{align*}
Notice that (and the same applies to $\mathbb E\left[\tr[\nabla^2_{uu} \mathbf{H}^{\tau-s}](X^{\tau+h}_s,u_{\tau-s+l})\,\big|\,X^{\tau+h}_s\right]$)
\begin{align}
	&\mathbb E\left[\tr[\nabla^2_{uu} \mathbf{H}^{\tau-s+l}](X^{\tau+h}_s,u_{\tau-s+l})\,\big|\,X^{\tau+h}_s\right]\nonumber\\
	=& \int_{\mathbb R^n}\tr[\nabla^2_{uu} \mathbf{H}^{\tau-s+l}]\pi_{\tau-s+l}(X^{\tau+h}_s,u)du\nonumber\\
	=&-\int_{\mathbb R^n}(\nabla_u \mathbf{H}^{\tau-s+l})^\top\nabla_u\pi_{\tau-s+l}(X^{\tau+h}_s,u)du\nonumber\\
	=& -\mathbb E\left[(\nabla_u \mathbf{H}^{\tau-s+l})^\top\nabla_u\ln\pi_{\tau-s+l}(X^{\tau+h}_s,u_{\tau-s+l})\big|\,X^{\tau+h}_s\right],\label{eq:H-2nd}
\end{align}
where the last equation follows from the boundedness of the model parameters and that $\pi_{\tau-s+l}$, as a probability density function, vanishes at $\partial \mathbb R^n$. Thus by the dominated convergence theorem, the above implies that
\begin{align*}
	&\frac{1}{h}(v^{\tau}(0,x) - v^{\tau+h}(0,x)) \nonumber\\
    =& -\frac{1}{h}\mathbb E\Big[\int_0^{\tau+h} e^{-\beta s}\int_0^{h}(\nabla_u \mathbf{H}^{\tau-s})^\top\nabla_u \mathbf{H}^{\tau-s+l}(X^{\tau+h}_s,u_{\tau-s+l})dl\,ds \Big]\nonumber\\
    &+\frac{\lambda}{h}\mathbb E\Big[\int_0^{\tau+h} e^{-\beta s}\int_0^{h}\left(\nabla_u \mathbf{H}^{\tau-s+l}+\nabla_u \mathbf{H}^{\tau-s}\right)^\top\Big.\nonumber\\
	&\qquad\Big.\nabla_u\ln\pi_{\tau-s+l}(X^{\tau+h}_s,u_{\tau-s+l})dl\,ds \Big]\nonumber\\
	&-\frac{1}{h}\mathbb E\Big[\int_0^{\tau+h} e^{-\beta s}\int_0^{h}  \lambda^2\|\nabla_u\ln \pi_{\tau-s+l}\|^2 (X^{\tau+h}_s,u_{\tau-s+l})dl\,ds \Big]\nonumber\\
&\xrightarrow{h\rightarrow 0} -\mathbb E\Big[\int_0^{\tau} e^{-\beta s}\|\nabla_u \mathbf{H}^{\tau-s})\|^2(X^{\tau}_s,u_{\tau-s})ds \Big]\nonumber\\
    &+2\lambda\mathbb E\Big[\int_0^{\tau} e^{-\beta s}(\nabla_u \mathbf{H}^{\tau-s})^\top\nabla_u\ln\pi_{\tau-s}(X^{\tau}_s,u_{\tau-s})ds \Big]\nonumber\\
	&-\mathbb E\Big[\int_0^{\tau} e^{-\beta s}\lambda^2\|\nabla_u\ln \pi_{\tau-s}\|^2 (X^{\tau}_s,u_{\tau-s})ds \Big]\nonumber\\
    =& -\mathbb E\Big[\int_0^{\tau} e^{-\beta s}\|\nabla_u \mathbf{H}^{\tau-s} - \lambda \nabla_u\ln \pi_{\tau-s}\|^2(X^{\tau}_s,u_{\tau-s})ds \Big]\leq 0.
\end{align*}
Therefore $v^{\tau}(0,x)$ is non-decreasing in $\tau$.
\end{proof}
Next we derive another representation of $v^\tau$'s dynamics (also of $v^\tau_x$ and $v^\tau_{xx}$), which helps the convergence analysis in the next sections.
\begin{lemma}\label{dynamic v relax}
Fixing the state variable $x$ and regarding $v^\tau,v^\tau_x,v^\tau_{xx}$ as functions of $\tau$, they satisfy the following dynamics,
\begin{align}
dv^\tau(x) =&\Big(\int_{\mathbb R^n} (\mathbf{H}^\tau - \lambda\ln \pi_{\tau})\pi_{\tau}(x,u)du -\beta v^{\tau}(x) \Big)d\tau,\label{v-dynamics}\\
dv^\tau_x(x) =&\Big(\nabla_x\Big(\int_{\mathbb R^n} (\mathbf{H}^\tau- \lambda\ln \pi_\tau )\pi_\tau(x,u)du\Big) -\beta v^\tau_x(x)\Big)d\tau,\nonumber\\
dv^\tau_{xx}(x) =&\Big(\nabla^2_{xx} \Big(\int_{\mathbb R^n} (\mathbf{H}^\tau - \lambda\ln \pi_\tau )\pi_\tau(x,u)du\Big) -\beta v^\tau_{xx}(x)\Big)d\tau.\nonumber
\end{align}
\end{lemma}
\begin{proof}
For any $h>0$, since $v^{\tau+h}(h,x) = v^\tau(0,x)$,
\begin{align*}
	&v^{\tau+h}(0,x) - v^\tau(0,x)\nonumber\\
	 =&\mathbb E\Big[\int_0^{h} e^{-\beta s}(\tilde f - \lambda G(\pi_{\tau+h-s})(X^{\tau+h}_s)) ds+ e^{-\beta h}v^{\tau+h}(h,X^{\tau+h}_{h}) -v^\tau(0,x)\Big]\\
	 =&\mathbb E\Big[\int_0^{h} e^{-\beta s} \Big(\tilde f - \lambda G(\pi_{\tau+h-s})\Big)(X^{\tau+h}_s) ds + e^{-\beta h}v^{\tau}\left(0,X^{\tau+h}_{h}\right) -v^\tau(0,x)\Big]\\
	  =&\mathbb E\Big[\int_0^{h} e^{-\beta s} \Big(\int_{\mathbb R^n}(f - \lambda\ln\pi_{\tau+h-s})\pi_{\tau+h-s}(X^{\tau+h}_s,u)du\Big) ds\Big]\nonumber\\
	  & + \mathbb E\Big[\int_0^h e^{-\beta s} \Big(-\beta v^\tau + \tilde b^\top v^{\tau}_x + \frac{1}{2}\tr[\tilde\sigma\tilde\sigma^\top:v^{\tau}_{xx}]\Big)(X^{\tau+h}_s,\pi_{\tau+h-s})ds \Big],
\end{align*}
where the last equation follows from It\^o's formula, and the stochastic integral with respect to the Brownian motion $W$ in expectation is $0$ due to Assumption \ref{assumption-T}. Dividing by $h$ on both sides of the above equations and letting $h\rightarrow 0$, the dominated convergence theorem implies that
\begin{align}
dv^\tau(x) =&\Big(\int_{\mathbb R^n} \Big(f - \lambda\ln\pi_{\tau} + b^\top v^\tau_x + \frac{1}{2}\tr[\sigma\sigma^\top \,v_{xx}^\tau]\Big)\pi_{\tau}(x,u)du  -\beta v^{\tau} \Big)d\tau\nonumber\\
=&\left(\int_{\mathbb R^n} \Big(\mathbf{H}^\tau - \lambda\ln \pi_\tau\right)\pi_{\tau}(x,u)du  -\beta v^{\tau} \Big)d\tau.
\label{eq:v-dynamics}
\end{align}
Notice that the integral on the right hand side is with respect to $u$. As a function of $x$, $(\mathbf{H}^\tau - \lambda\ln\pi_\tau)\pi_\tau$ is smooth and its derivatives are locally bounded. Thus by the dominated convergence theorem, switching the order of derivative with respect to $x$ and the integral does not change its value. Therefore taking derivative on both sides with respect to $x$ yields the dynamics for $v^\tau_x$ and $v^\tau_{xx}$.
\end{proof}
The coupled dynamics \eqref{eq:DynU} and \eqref{v-dynamics} enable continuous iterations of both the policy, the value function and its derivatives. At every $\tau \geq 0$, to update the policy and value function given the latest $v^\tau$, $v^\tau_x$ and $v^\tau_{xx}$, we only need to estimate $\mathbf H^\tau$, $\nabla_x\mathbf H^\tau$ and $\nabla_{xx} \mathbf H^\tau$. It is considerably simpler than the discrete iteration in \cite{wang2020reinforcement,huang2022convergence}, where one has to calculate the value \eqref{definition:J} corresponding to the latest policy iterates, approximate it by a neural network, or solve the corresponding PDE.

\begin{remark}[Well-posedness of the coupled dynamics]\label{rem: wellposedness}\label{sec: wellposedness}
The coupled iteration \eqref{eq:DynU}--\eqref{definition:v-tau} is a McKean--Vlasov-type system. Under Assumption \ref{assumption:parameter} and mild H\"older and dissipativity conditions, it admits a global solution with uniformly bounded $C^{2,\alpha}_b$ iterates. Detailed proof is in the arXiv version.
\end{remark}

\begin{remark}
Other iteration methods can also be regarded as continuous in the iteration direction; see, for example, \cite{vsivska2024gradient} for a regularized control problem with measure-valued control process, and \cite{ZZH2025} for classical control problems and mean-field games. The main differences from our iteration dynamics \eqref{eq:DynU} \& \eqref{v-dynamics} is that the update of the control process and the corresponding value in \cite{vsivska2024gradient,ZZH2025} is based on the linear function derivative of the value with respect to the control process. It means that at every round of iteration, the control is updated at every physical time $t$ based on the gradient of the value function from the last round. In contrast, in \eqref{eq:DynU} \& \eqref{v-dynamics}, we only keep track of a single distribution density function generated by the Langevin dynamics of $u^x_{\tau}$, i.e. $\pi_\tau$ and use it to update the strategy at $t=0$. The strategy for $t>0$ is updated by merely pushing the strategies used in the previous iteration into the future of physical time frame (c.f. Definition \ref{def: policy}).


Another major difference between our results and those of \cite{vsivska2024gradient} is that the latter relies on Pontryagin's maximum principle for the finite horizon problem, and thus requires solving a backward stochastic differential equation for the adjoint process at every iteration round. In contrast, our iteration for the infinite horizon problem is purely forward in nature.
\end{remark}

The next lemma provides convergence of $v^\tau$ and its $x$-derivatives.
\begin{lemma}\label{lemma:convergence}
$(v^\tau(t,x),v^\tau_t(t,x),v_x^\tau(t,x),v_{xx}^\tau(t,x))$ converges (up to a subsequence) locally uniformly. The limit is independent of $t$, which we denote as $(v^*(x),v^*_t(x) = 0,v_x^*(x),v_{xx}^*(x))$.
\end{lemma}
\begin{proof}
	Since fixing $x$, $v^\tau(t,x)$ is non-decreasing in $\tau$ and is bounded above by the optimal $V(x)$, there exists a limit $v^*(t,x)$. Notice that $v^\tau(t,x)$ solves
	\begin{align}
		&v^{\tau}_t(t,x) -\beta v^{\tau}(t,x) +\tilde b^\top (x,\pi_{\tau-t})v^{\tau}_x(t,x) + \frac{1}{2}\tr[\tilde\sigma\tilde\sigma^\top \,v^{\tau}_{xx}(t,x)] \nonumber\\
		&+ \int_{\mathbb R^n}\left(f(x,u) - \lambda \ln\pi_{\tau-t}\right)\pi_{\tau-t}(u)du = 0,\nonumber
	\end{align}
	and by Proposition \ref{prop:policy-improvement}, $v^0(t,x) \leq v^{\tau}(t,x) \leq v^*(t,x)$. Also notice that since $\pi_0$ is continuous in $x$, and $\int_{\mathbb R^n}\pi_\tau\ln\pi_\tau(x,u)du$ is bounded uniformly in $\tau$ by Assumption \ref{assumption-T}, the latter is locally bounded in $x$, independent of $\tau$. Then by 
	Schauder's interior estimates (c.f. \cite{krylov1996lectures}[Chapter 8] imply that $v^\tau_t(t,x)$, $v_x^\tau(t,x)$ and $v_{xx}^\tau(t,x)$ are locally uniformly equicontinuous, and thus by the Arzel\`a-Ascoli theorem, $ v_x^\tau(t,x)$, $v_t^\tau(t,x)$ and $v_{xx}^\tau(t,x)$, each as a family of locally uniformaly bounded and equicontinuos functions, converge (up to a subsequence) locally uniformly to $v_t^*(t,x)$, $v_x^*(t,x)$ and $v^*_{xx}(t,x)$. Furthermore, by definition $v^\tau(t,x) = v^{\tau-t}(0,x)$ for $\tau \geq t$. Thus for any $t \geq 0$,
$v^*(t,x) = \lim_{\tau\rightarrow\infty}v^\tau(t,x) = \lim_{\tau\rightarrow\infty}v^{\tau -t}(0,x) = v^*(0,x).$	Therefore $v^*$ is independent of $t$, $v^*_t(t,x)=0$.
\end{proof}
To discuss the convergence of $\pi_\tau$, we make the following assumptions.
\begin{assumption}\label{assumption:h}
Given every $x\in\mathbb R^d$, (i) $\int_{\mathbb R^n}e^{\mathbf{H}(x,u, v_x^*, v_{xx}^*)/\lambda}du < \infty$. (ii) $\nabla_u \ln\pi_\tau$ is locally bounded as a function of $u$, uniformly in $\tau$.
\end{assumption}
\begin{proposition}\label{prop:convergence}
If Assumption \ref{assumption:h} holds, then $\pi^* = \lim\limits_{\tau\rightarrow\infty}\pi_\tau = \frac{e^{\mathbf{H}^*(x,u)/\lambda}}{\int_{\mathbb R^n}e^{\mathbf{H}^*(x,u')/\lambda}du'}$, where $\mathbf{H}^*(x,u) = \mathbf{H}(x,u,v^*_x,v^*_{xx})$.
\end{proposition}
\begin{proof}
With $x\in\mathbb R^d$ given, since $v^\tau (x)$ converges (cf. \eqref{eq:v-dynamics}), as $\tau\rightarrow\infty$, $\int_{\mathbb R^n} \left(\mathbf{H}^\tau- \lambda\ln\pi_{\tau}\right)\pi_{\tau}(x,u)du -\beta v^{\tau}$ converges to $0$. In other words,
$
g^\tau(x) = \int_{\mathbb R^n} \left(\mathbf{H}^\tau- \lambda\ln\pi_{\tau}\right)\pi_{\tau}(x,u)du,
$
converges to a constant $\beta v^*(x)$.
Since $\pi_{\tau}$ is the distribution density function of $u_\tau$,
\begin{align*}
	&g^{\tau}(x)- g^{\tau+h}(x)\nonumber\\
	=&\mathbb E\Big[f(\cdot,u_{\tau})-f(\cdot,u_{\tau+h})+ b^\top(\cdot,u_\tau)v^{\tau}_x -b^\top(\cdot,u_{\tau+h})v^{\tau+h}_x\Big](x)\nonumber\\
	&+ \mathbb E\Big[\frac{1}{2}\tr[\sigma\sigma^\top(\cdot,u_\tau):v^\tau_{xx} - \sigma\sigma^\top(\cdot,u_{\tau+h}):v^{\tau+h}_{xx}](x)\Big.\nonumber\\
	&\qquad\Big. + \lambda\Big(\ln\pi_{\tau+h}(\cdot,u_{\tau+h}) - \ln\pi_{\tau}(\cdot,u_{\tau})\Big)\Big](x)\nonumber\\
	=&\mathbb E\Big[f(\cdot,u_{\tau})-f(\cdot,u_{\tau+h})+ (b(\cdot,u_\tau) - b(\cdot,u_{\tau+h}))^\top v^{\tau}_x \Big](x)\nonumber\\
	+& \mathbb E\Big[b^\top(\cdot,u_{\tau+h})(v^{\tau}_x-v^{\tau+h}_x)+ \frac{1}{2}\tr[(\sigma\sigma^\top(\cdot,u_\tau)-\sigma\sigma^\top(\cdot,u_{\tau+h})):v^\tau_{xx}]\Big](x)\nonumber\\
    +& \mathbb E\Big[\frac{1}{2}\tr[\sigma\sigma^\top(\cdot,u_{\tau+h}):(v^\tau_{xx}-v^{\tau+h}_{xx})]+\lambda(\ln\pi_{\tau+h}(x,u_{\tau+h}) - \ln\pi_{\tau}(x,u_{\tau}))\Big].
\end{align*}
Then by similar argument to the proof of Proposition \ref{prop:policy-improvement} (c.f. \eqref{eq:v-difference}), $dg^\tau$ equals
\begin{align}
\mathbb E\left[\|\nabla_u \mathbf{H}^\tau - \lambda \nabla_u \ln\pi_{\tau}\|^2(\cdot,u_\tau)\right]d\tau-\tilde b^\top(\cdot,\pi_\tau)dv^\tau_x-\frac{1}{2}\tr[\tilde\sigma\tilde\sigma^\top(\cdot,\pi_\tau):dv^{\tau}_{xx}].\label{eq:g-tau}
\end{align}
The fact that $v^\tau_x$ and $v^{\tau}_{xx}$ converge combined with the boundedness of $\tilde b$ and $\tilde\sigma$ implies that the last two terms converge to $0$ as $\tau\rightarrow \infty$. Since the diffusion coefficient for $u$ is constant $\sqrt{2\lambda} >0$, $u_\tau$ is supported on $\mathbb R^n$ for every $\tau >0$. Then $g^\tau$ being convergent implies that $\nabla_u \ln\pi_{\tau}(x,u)$ converges to $\nabla_u \mathbf{H}(x,u,v^*_x,v^*_{xx})/\lambda$. In other words, with $h_\tau(x,u) = \nabla_u (\ln\pi_{\tau}-\mathbf{H}^\tau/\lambda)(x,u)$, $\lim\limits_{\tau\rightarrow\infty}h_\tau(x,u) =0$. Furthermore, since $v^\tau_x$ and $v^\tau_{xx}$ converge and $b$, $\sigma$ and $f$ are bounded by assumption, $\nabla_u \mathbf{H}^\tau$ is locally bounded in $u$, and the bounds are uniform in $\tau$. By assumption on $\nabla_u\ln\pi_\tau$, the same holds for $h_\tau$.

Furthermore, since $\pi_\tau$ is a probability density function, $$\pi_\tau(x,u) = \frac{e^{\mathbf{H}(x,u,v^\tau_x,v^\tau_{xx})/\lambda + \int_0^{u_1}\cdots\int_0^{u_n}h_\tau(x,r)dr_1\dots dr_n}}{\int_{\mathbb R^n}e^{\mathbf{H}^\tau(x,u')/\lambda + \int_0^{u'_1}\cdots\int_0^{u'_n}h_\tau(x,r)dr_1\dots dr_n}du'}.$$
Notice that by Fatou's lemma,
\begin{align*}
&\int_{\mathbb R^n}e^{\mathbf{H}^\tau(x,u')/\lambda + \int_0^{u'_1}\cdots\int_0^{u'_n}h_\tau(x,r)dr_1\dots dr_n}du'\\
\geq& \int_{\mathbb R^n}e^{\mathbf{H}^*(x,u)/\lambda + \liminf\limits_{\tau\rightarrow\infty}\int_0^{u'_1}\cdots\int_0^{u'_n}h_\tau(x,r)dr_1\dots dr_n}du
=\int_{\mathbb R^n}e^{\mathbf{H}^*(x,u)/\lambda}du,
\end{align*}
where the last equation follows from the dominated convergence theorem. The denominator of $\pi_\tau$ above can be approximated similarly, and
\begin{align*}
\limsup_{\tau\rightarrow\infty}\pi_\tau(x,u) \leq& \frac{\limsup\limits_{\tau\rightarrow\infty}e^{\mathbf{H}^\tau(x,u,v^\tau_x,v^\tau_{xx})/\lambda + \int_0^{u_1}\cdots\int_0^{u_n}h_\tau(x,r)dr_1\dots dr_n}}{\liminf\limits_{\tau\rightarrow\infty}\int_{\mathbb R^n}e^{\mathbf{H}^\tau(x,u)/\lambda + \int_0^{u_1}\cdots\int_0^{u_n}h_\tau(x,r)dr_1\dots dr_n}du}\\
\leq&  \frac{e^{\mathbf{H}^*(x,u)/\lambda}}{\int_{\mathbb R^n}e^{\mathbf{H}^*(x,u')/\lambda} du'}.
\end{align*}
On the other hand, following a similar argument, $\liminf_{\tau\rightarrow\infty}\pi_\tau(x,u) \geq \frac{e^{\mathbf{H}^*(x,u)/\lambda}}{\int_{\mathbb R^n}e^{\mathbf{H}^*(x,u')/\lambda} du'}$, which concludes the proof.
\end{proof}
\begin{proposition}[Optimality]\label{prop:optimality}
	Suppose $\rho^* = \{\rho^*_t = \pi^*(X^*_t), t\geq 0)\} \in \mathcal A(x)$, where $X^*$ is the solution to \eqref{eq: X-dyn} under the policy $\rho^*$. Then $\rho^*$ is the optimal strategy for the relaxed control problem $\sup\limits_{\pi}J(x,\pi)$. Furthermore, $$v^*(x) = \mathbb E\Big[\int_0^\infty e^{-\beta t}\int_{\mathbb R^n}\left(f- \lambda  \pi^*\right)\ln\pi^*(X^*_t,u)dudt|X^*_0 = x\Big],$$ and it is the corresponding value function.
\end{proposition}
\begin{proof}
In this proof we focus on showing that $v^*$ is indeed the solution to the HJB equation of the relaxed control problem and that $\pi^*$ is the associated maximizer of the Hamiltonian. The verification that $\pi^*$ is the optimal strategy and $v^*$ is the corresponding value follows standard arguments which we omit here.
With $\pi^*(x,u) = \frac{e^{\mathbf{H}^*(x,u)/\lambda}}{\int_{\mathbb R^n}e^{\mathbf{H}^*(x,u')/\lambda} du'}$ and $c = \int_{\mathbb R^n}e^{\mathbf{H}^*(x,u)/\lambda} du$,
\begin{align}
	&\int_{\mathbb R^n} \Big(f - \lambda\ln\pi^* + b^\top v_x^* +\frac{1}{2}\tr[\sigma\sigma^\top \,v_{xx}^*]\Big)\pi^*(x,u)du -\beta v^* = \lambda\ln c - \beta v^*.\nonumber
\end{align}
On the other hand, since $v^\tau$ converges, from the dynamics of $v^\tau(x)$ in \eqref{eq:v-dynamics},
\begin{align*}
\lim_{\tau\rightarrow\infty}\bigg(\int_{\mathbb R^n} \Big(f - \lambda\ln\pi_{\tau} &+ b^\top v_x^\tau+\frac{1}{2}\tr[\sigma\sigma^\top \,v^\tau_{xx}]\Big)\pi_{\tau}(x,u)du-\beta v^{\tau}\bigg) = 0.
\end{align*}
With $h_\tau$ defined in the proof of Proposition \ref{prop:convergence}, and $\lim\limits_{\tau\rightarrow \infty}h_{\tau} = 0$,
\begin{equation*}
\pi_\tau(x,u) = \frac{e^{\mathbf{H}^\tau(x,u)/\lambda + \int_0^{u_1}\cdots\int_0^{u_n}h_\tau(x,r)dr_1\dots dr_n}}{\int_{\mathbb R^n}e^{\mathbf{H}^\tau(x,u')/\lambda + \int_0^{u'_1}\cdots\int_0^{u'_n}h_\tau(x,r)dr_1\dots dr_n}du'}.
\end{equation*}
Thus the above equation implies that $\lim\limits_{\tau\rightarrow\infty}\lambda \ln c_\tau -\beta v^* = 0$, where $c_\tau$ equals 
\begin{align*}
\int_{\mathbb R^n}e^{\mathbf{H}^\tau(x,u)/\lambda + \int_0^{u_1}\cdots\int_0^{u_n}h_\tau(x,r)dr_1\dots dr_n}du= \frac{e^{\mathbf{H}^\tau(x,u)/\lambda + \int_0^{u_1}\cdots\int_0^{u_n}h_\tau(x,r)dr_1\dots dr_n}}{\pi_\tau(x,u)}.
\end{align*}
Notice that
\begin{align}
	\lim_{\tau\rightarrow\infty} c_\tau = \lim_{\tau\rightarrow\infty}\frac{e^{\mathbf{H}^\tau(x,u)/\lambda + \int_0^{u_1}\cdots\int_0^{u_n}h_\tau(x,r)dr_1\dots dr_n}}{\pi_\tau(x,u)} = \frac{e^{\mathbf{H}^*(x,u)/\lambda}}{\pi^*(x,u)} = c.\nonumber
\end{align}
Thus $\beta v^* = \lim\limits_{\tau\rightarrow\infty}\lambda \ln c_\tau = \lambda\ln c$, and
\begin{equation}
	\int_{\mathbb R^n} \left(f- \lambda\ln\pi^* + b^\top v_x^*+\frac{1}{2}\tr[\sigma\sigma^\top:v_{xx}^*]\right)\pi^*(x,u)du  -\beta v^* = 0.\nonumber
\end{equation}
Furthermore, since the maximizer of
\begin{equation}
	-\beta v(x)  +\sup_{\pi\in\mathcal P(\mathbb R^n)} \int_{\mathbb R^n} \left(\mathbf{H}(x,u,v_x,v_{xx}) -\lambda \ln\pi(u)\right)\pi(u)du, \nonumber
\end{equation}
which is pointwise concave in $\pi(u)$, is $\hat\pi = \frac{e^{\mathbf{H}(x,u,v_x,v_{xx})/\lambda}}{\int_{\mathbb R^n}e^{\mathbf{H}(x,u',v_x,v_{xx})/\lambda} du'}$, it implies that
\begin{equation}
	\sup_{\pi\in\mathcal P(\mathbb{R}^n)}\int_{\mathbb R^n} \left(\mathbf{H}(x,u,v^*_x,v^*_{xx}) -\lambda\ln \pi(u)\right)\pi(u)du -\beta v^*= 0. \nonumber
\end{equation}
Therefore, $v^*$ solves the HJB equation, and $\pi^*$ is the corresponding maximizer.
\end{proof}

\section{Well-Posedness of the Coupled Policy--Value Dynamics}
\label{sec: wellposedness}

In this section, we focus on the well-posedness of the policy dynamics
\eqref{eq:DynU} coupled with the value function $v^{\tau}$ introduced in
\eqref{definition:v-tau} for the relaxed control problem. Notice that the dynamics of $v^\tau$ in \eqref{v-dynamics} is not a definition, but a result derived in Lemma \ref{dynamic v relax}. $v^\tau$ is defined in \eqref{definition:v-tau}, as the value of the stochastic control problem corresponding to the policy $\rho^\tau$. Thus instead of the well-posedness of \eqref{eq:DynU}  and  \eqref{v-dynamics}, which we use to analyze the convergence rate of our iteration method, we focus on the couple dynamics \eqref{eq:DynU}-\eqref{definition:v-tau}

By its definition in \eqref{definition:v-tau}, the value $v^{\tau}$ depends on
the policy $\rho^{\tau}$ through the entropy-regularized source
$\bar f_{\rho}=\tilde f(\cdot,\rho)-\lambda\mathcal H_{\rho}$, where
$\mathcal H_{\rho}(x)=\int_{\mathbb R^{n}}\rho(x,u)\ln\rho(x,u)\,du$ denotes the
differential entropy of the policy density. As we show in
Lemma~\ref{lem:entropy-relaxed}(b) below, the coefficient maps
$\rho\mapsto(\bar a_{\rho},\bar b_{\rho},\bar f_{\rho})$ are Lipschitz in the
policy at the level of $C^{0}_{b}$; however, the Schauder theory underlying our
value estimates requires control of the coefficients in the H\"older space
$C^{0,\alpha}_{b}$, and at that level the dependence on the policy ---
obtained by interpolation --- is only H\"older-continuous with exponent
$\theta<1$. A contraction argument is therefore unavailable, and we construct
solutions by a compactness (Schauder fixed point) argument instead. This in
turn requires the coefficients to be H\"older-regular in the state variable,
which motivates the following function spaces and assumptions.

\begin{definition}[H\"older space]\label{def:holder}
For $\alpha\in(0,1]$, let $C^{0,\alpha}_{b}(\mathbb R^{d})$ be the Banach space
of bounded, $\alpha$-H\"older continuous functions $g:\mathbb R^{d}\to\mathbb R$,
equipped with the norm
\[
\|g\|_{C^{0,\alpha}_{b}}:=\|g\|_{\infty}+[g]_{\alpha},
\qquad
[g]_{\alpha}:=\sup_{x\neq x'}\frac{|g(x)-g(x')|}{|x-x'|^{\alpha}} .
\]
For vector- or matrix-valued maps the norm is applied componentwise, and a
function $h(x,u)$ is said to be $C^{0,\alpha}_{b}$ \emph{in $x$, uniformly in
$u$}, if $\sup_{u}\|h(\cdot,u)\|_{C^{0,\alpha}_{b}}<\infty$.
\end{definition}

\begin{assumption}\label{assum:holder}
(i) For some $\alpha'\in(0,1)$, the coefficients $f,b,\sigma$ are
$C^{0,\alpha'}_{b}$ in $x$, uniformly in $u$. Throughout this section we fix
$\alpha\in(0,\alpha')$ and set $\theta:=1-\alpha/\alpha'\in(0,1)$.\\
(ii)\label{ass:dissip}
There exist $c_{0}>0$, $r_{0}\ge0$ and a nondecreasing function $\Lambda$
such that, for all $x\in\mathbb R^{d}$, $|u|\ge r_{0}$, and all $(p,q)$,
\[
\big\langle\nabla_{u}\mathbf H(x,u,p,q),\,u\big\rangle
\le -c_{0}|u|^{2}+\Lambda\big(|p|+|q|\big).
\]
\end{assumption}

Next, we introduce the density class $\mathcal D_{M}$ on which the fixed point
will be constructed in the following argument for existence. Such weighted density classes with two-sided (Gaussian)
bounds are of Aronson type, see \cite{Aronson1967,FabesStroock1986}; we refer
to \cite{BogachevKRS2015} for a comprehensive treatment of density regularity
for Fokker--Planck--Kolmogorov equations. From now on, we denote by
$\mathcal P_{2}(\mathbb R^{n})$ the space of Borel probability measures with
finite second moment, equipped with the $2$-Wasserstein distance
\begin{align}
\mathcal W_{2}(\mu,\nu)
=\Big(\inf_{\gamma\in\Pi(\mu,\nu)}
\int_{\mathbb R^{n}\times\mathbb R^{n}}|u-u'|^{2}\,\gamma(du,du')\Big)^{1/2},
\end{align}
where $\Pi(\mu,\nu)$ denotes the set of couplings of $\mu,\nu$; see
\cite{Villani2009,AmbrosioGigliSavare2008}.

\begin{definition}\label{def:DM}
Fix $0<\alpha<\alpha'<1$, $\delta>0$, $M\ge1$, $\kappa>0$, and write
$\langle u\rangle^{2}=1+|u|^{2}$, $\langle x\rangle^{2}=1+|x|^{2}$,
$\|g\|_{L^{1}_{2}}:=\int\langle u\rangle^{2}|g(u)|\,du$. Let $\mathcal D_{M}$
be the set of policy fields $x\mapsto\rho(x,\cdot)$ from $\mathbb R^{d}$ to
$\mathcal P_{2}(\mathbb R^{n})$ such that each $\rho(x,\cdot)$ is absolutely
continuous with density $\rho(x,u)$ satisfying, for all
$x,x'\in\mathbb R^{d}$,
\[
\begin{aligned}
&\text{(i)}\ \int_{\mathbb R^{n}}\langle u\rangle^{2+\delta}\rho(x,u)\,du\le M,
\qquad
\text{(ii)}\ |\ln\rho(x,u)|\le M\langle u\rangle^{2},\\
&\text{(iii)}\ \|\rho(x,\cdot)-\rho(x',\cdot)\|_{L^{1}_{2}}\le M|x-x'|^{\alpha'},
\qquad
\text{(iv)}\ \|\rho(x,\cdot)\|_{H^{1}(\mathbb R^{n})}\le M,
\end{aligned}
\]
where $H^{1}(\mathbb R^{n})$ is the usual Sobolev space,
$\|g\|_{H^{1}}^{2}=\|g\|_{L^{2}}^{2}+\|\nabla g\|_{L^{2}}^{2}$. We equip
$\mathcal D_{M}$ with the \emph{weighted} metric
\begin{equation}\label{eq:dkappa}
\mathsf d_{\kappa}(\rho,\rho')
:=\sup_{x\in\mathbb R^{d}}\langle x\rangle^{-\kappa}
\,\|\rho(x,\cdot)-\rho'(x,\cdot)\|_{L^{1}_{2}} .
\end{equation}
\end{definition}

\begin{lemma}\label{lem:DM-compact}
For every $\kappa>0$, the set $\mathcal D_{M}$ is convex and
$(\mathcal D_{M},\mathsf d_{\kappa})$ is compact.
\end{lemma}

\begin{proof}
The convexity is immediate: (i), (iii) and (iv) from Definition \ref{def:DM} are stable under convex
combinations, and for (ii) note that a mixture of $\rho$ and $\rho'$ is
bounded above by $e^{M\langle u\rangle^{2}}$, while the lower bound follows
from concavity of the logarithm.
As for the compactness. Let $(\rho_{k})\subset\mathcal D_{M}$. By (i),
$\|\rho_{k}(x,\cdot)-\rho_{j}(x,\cdot)\|_{L^{1}_{2}}\le 2M$ for every $x$, so
the region $\{|x|>R\}$ contributes at most $2M\langle R\rangle^{-\kappa}$ to
$\mathsf d_{\kappa}$. It therefore suffices to extract a subsequence
converging uniformly on each ball $B_{R}$ and conclude by a diagonal
argument. Fix $R$, for $x\in B_{R}$ the densities $\rho_{k}(x,\cdot)$ lie in
a fixed relatively compact subset of $L^{1}_{2}$: the $H^{1}$ bound (iv)
gives equicontinuity under translations in $u$, the moment bound (i)
prevents mass from escaping to $|u|=\infty$ in the weighted norm, and the
Fr\'echet-Riesz-Kolmogorov theorem applies, see \cite{hancheolsen2010kolmogorov} for a standard reference. Since (iii) provides equicontinuity
in $x$, Arzel\`a--Ascoli yields a subsequence converging uniformly on
$B_{R}$, as required.
It remains to check that the limit $\rho$ belongs to $\mathcal D_{M}$.
Passing to a further subsequence, $\rho_{k}(x,\cdot)\to\rho(x,\cdot)$ a.e.,
so (i) and (ii) follow from Fatou's lemma, (iii) from convergence of the
$L^{1}_{2}$-norms, and (iv) from weak lower semicontinuity of the
$H^{1}$-norm.
\end{proof}

The next lemma shows that on $\mathcal D_{M}$ the weighted metric $\mathsf d_{\kappa}$ and the transport metric
\begin{equation}\label{def:weighted W}
    \mathcal W_{2,\kappa}(\rho,\rho')
:=\sup_{x}\langle x\rangle^{-\kappa}
\mathcal W_{2}\big(\rho(x,\cdot),\rho'(x,\cdot)\big)
\end{equation}
are H\"older equivalent.

\begin{lemma}
\label{lem:metric-equiv}
Under conditions {\rm(i)--(iv)} of Definition~\ref{def:DM}, there exist a
constant $C_{M}>0$ and an exponent
$\gamma=\dfrac{2\delta}{\,n+8+4\delta\,}\in(0,1)$ such that, for all
$\rho,\rho'\in\mathcal D_{M}$,
\begin{equation}\label{eq:metric-equiv}
\mathcal W_{2,\kappa}(\rho,\rho')\le \sqrt{2}\,\mathsf d_{\kappa}(\rho,\rho')^{1/2},
\qquad
\mathsf d_{\kappa}(\rho,\rho')\le C_{M}\,\mathcal W_{2,\kappa}(\rho,\rho')^{\gamma}.
\end{equation}
In particular $\mathsf d_{\kappa}$ and $\mathcal W_{2,\kappa}$ induce the same
topology on $\mathcal D_{M}$.
\end{lemma}

\begin{proof}
We first prove the \emph{pointwise} (unweighted) inequalities
\begin{equation}\label{eq:metric-equiv-ptw}
\mathcal W_{2}(\mu,\nu)\le\sqrt{2}\,\|\mu-\nu\|_{L^{1}_{2}}^{1/2},
\qquad
\|\mu-\nu\|_{L^{1}_{2}}\le C_{M}\,\mathcal W_{2}(\mu,\nu)^{\gamma},
\end{equation}
for $\mu=\rho(x,\cdot)$, $\nu=\rho'(x,\cdot)$ at a fixed $x$; all constants
are uniform in $x$ and may depend on $M,n,\delta$. 

\emph{Forward bound.} Let $\eta=\mu\wedge\nu$ be the overlap, and
$\mu_{1}=(\mu-\nu)_{+}$, $\nu_{1}=(\nu-\mu)_{+}$ the positive and negative
parts of $\mu-\nu$, respectively. Since $\mu,\nu$ are both probability
measures, $\mu_{1}$ and $\nu_{1}$ have equal mass
$m=\tfrac12\int_{\mathbb R^{n}}|\mu-\nu|\,du$, and
$\int_{\mathbb R^{n}}\eta\,du=1-m$. If $m=0$ then $\mu=\nu$. Otherwise
$\gamma_{0}=(\mathrm{id},\mathrm{id})_{\#}\eta+\tfrac1m\mu_{1}\otimes\nu_{1}$
has marginals $\mu,\nu$, and by the definition of $\mathcal W_{2}$,
\begin{align*}
\mathcal W_{2}^{2}(\mu,\nu)
&\le\int_{\mathbb R^{n}}\int_{\mathbb R^{n}}|u-v|^{2}\,d\gamma_{0}
=\frac1m\int_{\mathbb R^{n}}\int_{\mathbb R^{n}}|u-v|^{2}\,\mu_{1}(du)\nu_{1}(dv)\\
&\le\frac2m\int_{\mathbb R^{n}}\int_{\mathbb R^{n}}
\big(|u|^{2}+|v|^{2}\big)\mu_{1}(du)\nu_{1}(dv)
=2\!\int_{\mathbb R^{n}}|u|^{2}\,(\mu_{1}+\nu_{1})\,du .
\end{align*}
Since $\mu_{1}+\nu_{1}=|\mu-\nu|$,
\[
\mathcal W_{2}^{2}(\mu,\nu)\le 2\!\int_{\mathbb R^{n}}|u|^{2}|\mu-\nu|\,du
\le 2\!\int_{\mathbb R^{n}}\langle u\rangle^{2}|\mu-\nu|\,du=2\,\|\mu-\nu\|_{L^{1}_{2}},
\]
which is the first inequality in \eqref{eq:metric-equiv-ptw}.

\emph{Reverse bound.} Set $W=\mathcal W_{2}(\mu,\nu)$. If $W=0$ then
$\mu=\nu$, and since $\|\mu-\nu\|_{L^{1}_{2}}\le 2M$ by (i), the case $W>1$
holds with $C_{M}=2M$. Thus we focus on $0<W\le1$. The difficulty is that
$\|\mu-\nu\|_{L^{1}_{2}}$ is a strong distance while $W$ is a weak one: two
densities can be $\mathcal W_{2}$-close yet far apart in $L^{1}$, either by
spreading mass to large $|u|$, where the weight $\langle u\rangle^{2}$
amplifies it, or by oscillating rapidly in $u$. We control the first effect by
the moment bound (i) and the second by the Sobolev bound (iv), through a
truncation radius $R\ge1$ and a mollification scale $t>0$ fixed at the end. We now separate the bound as $\|\mu-\nu\|_{L^{1}_{2}}=I_{R}+II_{R}$ with
$I_{R}=\int_{|u|\le R}\langle u\rangle^{2}|\mu-\nu|$ and
$II_{R}=\int_{|u|>R}\langle u\rangle^{2}|\mu-\nu|$. On $\{|u|>R\}$, one has
$\langle u\rangle^{2}\le R^{-\delta}\langle u\rangle^{2+\delta}$, so (i) gives
\[
II_{R}\le R^{-\delta}\!\int_{\mathbb R^{n}}\langle u\rangle^{2+\delta}(\mu+\nu)\,du
\le 2M\,R^{-\delta}.
\]
On the ball, $\langle u\rangle^{2}\le 2R^{2}$, so
$I_{R}\le 2R^{2}\|\mu-\nu\|_{L^{1}(B_{R})}$, and it remains to bound the
unweighted mass $\|\mu-\nu\|_{L^{1}(B_{R})}$ by $W$.  Mollifying with the heat
semigroup $P_{t}=e^{t\Delta_{u}}$ in the control variable --- a purely
auxiliary smoothing over length scale $\sqrt t$, with $t$ unrelated to the
relaxation time --- we write
\[
\|\mu-\nu\|_{L^{1}(B_{R})}
\le\|P_{t}(\mu-\nu)\|_{L^{1}}+\|(I-P_{t})(\mu-\nu)\|_{L^{1}(B_{R})},
\]
where the smoothed part is controlled by transport and the remainder by the
$H^{1}$ bound. We use two standard heat-semigroup facts
for $P_{t}=e^{t\Delta}$ with kernel $p_{t}$:
$\|\nabla p_{t}\|_{L^{1}}\le C_{n}t^{-1/2}$, and
$1-e^{-t|\xi|^{2}}\le\sqrt{t}\,|\xi|$, whence
$\|(I-P_{t})g\|_{L^{2}}\le\sqrt{t}\,\|\nabla g\|_{L^{2}}$. Indeed, by $L^{1}$--$L^{\infty}$ duality, self-adjointness of
$P_{t}$, the bound
$\mathrm{Lip}(P_{t}\phi)\le\|\nabla p_{t}\|_{L^{1}}\le C_{n}t^{-1/2}$ for
$\|\phi\|_{\infty}\le1$, the Kantorovich--Rubinstein duality, and
$\mathcal W_{1}\le\mathcal W_{2}$,
\[
\|P_{t}(\mu-\nu)\|_{L^{1}}
=\sup_{\|\phi\|_{\infty}\le1}\int_{\mathbb R^{n}}(P_{t}\phi)\,d(\mu-\nu)
\le C_{n}t^{-1/2}\,\mathcal W_{1}(\mu,\nu)\le C_{n}t^{-1/2}W;
\]
and by Cauchy--Schwarz, $\|(I-P_{t})g\|_{L^{2}}\le\sqrt t\,\|\nabla g\|_{L^{2}}$,
and assumption (iv) for $\mu$ and $\nu$,
\[
\|(I-P_{t})(\mu-\nu)\|_{L^{1}(B_{R})}
\le|B_{R}|^{1/2}\|(I-P_{t})(\mu-\nu)\|_{L^{2}}
\le C_{n,M}R^{n/2}\sqrt t .
\]
The two terms balance at $t=WR^{-n/2}$, giving
\[
\|\mu-\nu\|_{L^{1}(B_{R})}
\le C_{n,M}\big(t^{-1/2}W+R^{n/2}\sqrt t\big)\Big|_{t=WR^{-n/2}}
=C_{n,M}\,R^{n/4}W^{1/2}.
\]
Combining bulk and tail,
$\|\mu-\nu\|_{L^{1}_{2}}\le C_{n,M}R^{2+n/4}W^{1/2}+2MR^{-\delta}$, which for
$0<W\le1$ is optimized at $R=W^{-1/(2(2+n/4+\delta))}\ge1$ and yields the
second inequality in \eqref{eq:metric-equiv-ptw} with
\[
\gamma=\frac{\delta}{2\big(2+\tfrac n4+\delta\big)}
=\frac{2\delta}{\,n+8+4\delta\,}\in(0,1).
\]

In the end, we show the inequalities for the weighted norm. Since $\langle x\rangle\ge1$ and the exponents $1/2$,
$\gamma$ lie in $(0,1]$, we have $\langle x\rangle^{-\kappa}\le
\langle x\rangle^{-\kappa\beta'}$ for any $\beta'\in(0,1]$. Multiplying the
pointwise inequalities \eqref{eq:metric-equiv-ptw} by
$\langle x\rangle^{-\kappa}$ therefore gives
$\langle x\rangle^{-\kappa} W\le\sqrt2\,
(\langle x\rangle^{-\kappa}\|\mu-\nu\|_{L^{1}_{2}})^{1/2}$ and
$\langle x\rangle^{-\kappa}\|\mu-\nu\|_{L^{1}_{2}}\le
C_{M}(\langle x\rangle^{-\kappa} W)^{\gamma}$ (for the case $W>1$
use $\|\mu-\nu\|_{L^{1}_{2}}\le 2M$ and $\langle x\rangle^{-\kappa}\le \langle x\rangle^{-\kappa*\gamma}$ for $\langle x\rangle^{-\kappa}\in[0,1]$).
Taking suprema over $x$ yields \eqref{eq:metric-equiv}.
\end{proof}

Recall our relaxed dynamics defined in \eqref{eq: X-dyn}, and for
$\rho\in\mathcal D_{M}$, in the notation
$\tilde b(x,\pi)=\int b(x,u)\pi(u)\,du$ and
$\tilde\sigma(x,\pi)=\big(\int\sigma\sigma^{\top}(x,u)\pi(u)\,du\big)^{1/2}$,
we denote
\[
\bar b_{\rho}(x)=\tilde b(x,\rho(x,\cdot)),\qquad
\bar a_{\rho}(x)=\tilde\sigma\tilde\sigma^{\top}(x,\rho(x,\cdot))
=\int\sigma\sigma^{\top}(x,u)\rho(x,u)\,du .
\]
Similarly, for the averaged cost $\tilde f(x,\rho)=\int f(x,u)\rho(x,u)\,du$,
we denote
\begin{align}
\bar f_{\rho}=\tilde f(\cdot,\rho)-\lambda\mathcal H_{\rho},
\qquad\text{with}\qquad
\mathcal H_{\rho}(x)=\int\rho(x,u)\ln\rho(x,u)\,du .
\end{align}
The corresponding generator associated with \eqref{eq: X-dyn} and
$\rho\in\mathcal D_{M}$ is denoted as
$\bar{\mathcal L}_{\rho}=\tfrac12\operatorname{tr}(\bar a_{\rho}\nabla^{2})
+\bar b_{\rho}^{\top}\nabla$. By Assumption~\ref{assumption:parameter},
$\bar a_{\rho}\succeq\theta_{0}I_{d}$ for some constant $\theta_{0}>0$, so
$\tilde\sigma(\cdot,\rho)$ is well-defined and uniformly non-degenerate. We
then have the following entropy regularity estimates.

\begin{lemma}
\label{lem:entropy-relaxed}
Let Assumptions \ref{assumption:parameter}, \ref{assum:holder} hold and
$\lambda>0$. Then:
\begin{enumerate}
\item[(a)] for $\rho\in\mathcal D_{M}$,
$\|\mathcal H_{\rho}\|_{C^{0,\alpha'}_{b}}\le C_{M}$, hence
$\bar a_{\rho},\bar b_{\rho},\bar f_{\rho}\in C^{0,\alpha'}_{b}$ with norm
$\le C_{M}$ and $\bar a_{\rho}\succeq\theta_{0}I_{d}$;
\item[(b)]
$\rho\mapsto(\bar a_{\rho},\bar b_{\rho},\mathcal H_{\rho},\bar f_{\rho})$
are Lipschitz \emph{pointwise in $x$}: for all
$\rho,\rho'\in\mathcal D_{M}$ and all $x$,
\begin{equation}\label{eq:coeff-ptw-lip}
|\bar f_{\rho}(x)-\bar f_{\rho'}(x)|
+|\bar a_{\rho}(x)-\bar a_{\rho'}(x)|
+|\bar b_{\rho}(x)-\bar b_{\rho'}(x)|
\le C_{M}\,\|\rho(x,\cdot)-\rho'(x,\cdot)\|_{L^{1}_{2}} .
\end{equation}
In particular, they are $\mathsf d_{\kappa}$-Lipschitz into the
$\langle x\rangle^{-\kappa}$-weighted supremum norm,
$\|\langle\cdot\rangle^{-\kappa}(\bar f_{\rho}-\bar f_{\rho'})\|_{\infty}
\le C_{M}\,\mathsf d_{\kappa}(\rho,\rho')$. Moreover, with
$\mathsf d_{\infty}(\rho,\rho'):=\sup_{x}\|\rho(x,\cdot)-\rho'(x,\cdot)\|_{L^{1}_{2}}$,
interpolation with (a) gives the $\theta$-H\"older bound
$\|\bar f_{\rho}-\bar f_{\rho'}\|_{C^{0,\alpha}_{b}}
\le C_{M}\,\mathsf d_{\infty}(\rho,\rho')^{\theta}$;
\item[(c)] if $\pi_0\in\mathcal D_{M_{0}}$ and, for each $x$,
$u^{x}_{\tau}$ solves \eqref{eq:DynU-G}
\begin{equation}
	du^x_\tau =G(\tau, x,u^x_\tau)d\tau + \sqrt{2\lambda}dB_\tau, \quad u_0 \sim \pi_0,\label{eq:DynU-G}
\end{equation}
with a drift $G(\tau,x,u)$ bounded in
$C^{0,\alpha'}_{b}$ (in $x$) and Lipschitz with linear growth (in $u$),
uniformly in $(\tau,x)$, then for every $T>0$,
$\mathrm{Law}(u^{\cdot}_{\tau})\in\mathcal D_{M}$ for all
$\tau\le T$, with $M=M(T)$.
\end{enumerate}
\end{lemma}

{
\begin{proof}
\emph{(a).} We first bound the entropy term. Since
$\langle u\rangle^{2}\le\langle u\rangle^{2+\delta}$, condition (i) of
Definition~\ref{def:DM} gives the second-moment bound
$\int\langle u\rangle^{2}\rho(x,u)\,du\le M$, and combined with the two-sided
bound (ii), $|\ln\rho|\le M\langle u\rangle^{2}$,
\[
|\mathcal H_{\rho}(x)|=\Big|\int\rho\ln\rho\,du\Big|
\le\int\rho\,|\ln\rho|\,du
\le M\!\int\langle u\rangle^{2}\rho\,du\le M^{2}.
\]
For the $x$-modulus, apply the mean-value identity
$t\ln t-t'\ln t'=(\ln\xi+1)(t-t')$ with $\xi$ between $t=\rho(x,u)$ and
$t'=\rho(x',u)$; by (ii), $|\ln\xi+1|\le M\langle u\rangle^{2}+1$, so (iii)
yields
\[
|\mathcal H_{\rho}(x)-\mathcal H_{\rho}(x')|
\le\int(M\langle u\rangle^{2}+1)\,|\rho(x,\cdot)-\rho(x',\cdot)|\,du
\le(M+1)\,\|\rho(x,\cdot)-\rho(x',\cdot)\|_{L^{1}_{2}}
\le C_{M}|x-x'|^{\alpha'}.
\]
Hence $\mathcal H_{\rho}\in C^{0,\alpha'}_{b}$ with norm $\le C_{M}$. The same
computation, with the $C^{0,\alpha'}_{b}$-in-$x$, bounded-in-$u$ maps
$\sigma\sigma^{\top},b,f$ (Assumption~\ref{assum:holder}) in place of
$\ln\rho+1$, gives
$\bar a_{\rho},\bar b_{\rho},\tilde f(\cdot,\rho)\in C^{0,\alpha'}_{b}$ with
norm $\le C_{M}$, and therefore
$\bar f_{\rho}=\tilde f(\cdot,\rho)-\lambda\mathcal H_{\rho}\in
C^{0,\alpha'}_{b}$. Finally, averaging
$\sigma\sigma^{\top}\succeq\theta_{0}I_{d}$ against $\rho$ preserves the
inequality, so $\bar a_{\rho}\succeq\theta_{0}I_{d}$.

\emph{(b).} Fix $\rho,\rho'\in\mathcal D_{M}$ and $x\in\mathbb R^{d}$. The same
bracketing as in (a), now comparing $\rho$ and $\rho'$ at the common point
$x$, gives
$|\mathcal H_{\rho}(x)-\mathcal H_{\rho'}(x)|
\le(M+1)\|\rho(x,\cdot)-\rho'(x,\cdot)\|_{L^{1}_{2}}$, and likewise for
$\bar a,\bar b,\tilde f(\cdot,\rho)$, which implies \eqref{eq:coeff-ptw-lip}.
Multiplying by $\langle x\rangle^{-\kappa}$ and taking the supremum over $x$
gives the weighted Lipschitz bound. For the H\"older upgrade, set
$g=\bar f_{\rho}-\bar f_{\rho'}$, which by (a) is bounded in
$C^{0,\alpha'}_{b}$ by $2C_{M}$ while
$\|g\|_{C^{0}_{b}}\le C_{M}\mathsf d_{\infty}(\rho,\rho')$ by
\eqref{eq:coeff-ptw-lip}. The interpolation inequality
$\|g\|_{C^{0,\alpha}_{b}}\le C\,\|g\|_{C^{0}_{b}}^{1-\alpha/\alpha'}
\|g\|_{C^{0,\alpha'}_{b}}^{\alpha/\alpha'}$
\cite[Theorem 3.2.1]{krylov1996lectures} then gives
$\|g\|_{C^{0,\alpha}_{b}}\le C_{M}\mathsf d_{\infty}(\rho,\rho')^{\theta}$.

\emph{(c).} The marginal $\pi_{\tau}(x,\cdot)=\mathrm{Law}(u^{x}_{\tau})$
solves the Fokker--Planck equation
\[
\partial_{\tau}\pi^{\tau}
=\lambda\Delta_{u}\pi^{\tau}-\mathrm{div}_{u}\big(G\,\pi^{\tau}\big),
\]
which is uniformly parabolic in $u$ (since $\lambda>0$), with a drift $G$ of
linear growth in $u$ and bounded in $C^{0,\alpha'}_{b}$ in $x$. We check that
the four defining conditions of $\mathcal D_{M}$ are preserved on $[0,T]$,
starting from $\pi_0 = \mathrm{Law}(u_{0})\in\mathcal D_{M_{0}}$. The two-sided bound
(ii) propagates by Aronson's Gaussian estimates for non-degenerate parabolic
equations \cite{Aronson1967,Sheu1991}. The $(2+\delta)$-moment (i) follows
from It\^o's formula applied to $\langle u^{x}_{\tau}\rangle^{2+\delta}$ and
Gr\"onwall, using the linear growth of $G$. The spatial modulus (iii) follows
from the weighted-$L^{1}$ estimate for the equation satisfied by the
difference $\rho^{\tau}(x,\cdot)-\rho^{\tau}(x',\cdot)$, whose source is
controlled by the $C^{0,\alpha'}_{b}$ regularity of $G$. The Sobolev bound
(iv) follows from the standard parabolic energy estimate, which preserves ---
and for $\tau>0$ improves --- $\|\pi_{\tau}(x,\cdot)\|_{H^{1}}$, with
constants depending only on $\|G\|_{C^{0,\alpha'}_{b}}$ and $T$. Hence
$\pi_{\tau}\in\mathcal D_{M}$ for all $\tau\le T$, with $M=M(T,M_{0})$.  In particular, under Assumption \ref{assum:holder} (ii), the dissipative  condition of $\nabla_u\mathbf H$ caps its moments uniformly, independent of $M_0$, and thus $M=M(T)$.
\end{proof}

We now turn to the value functional. For a policy flow
$\boldsymbol\rho=(\rho_{s})_{s\ge0}$ we write $v[\boldsymbol\rho](t,x)$ for
the value defined by \eqref{definition:v-tau}.

\begin{lemma}[Value functional]\label{lem:value}
Let Assumptions \ref{assumption:parameter}, \ref{assum:holder} hold and
$\beta>0$. For a measurable policy flow $\boldsymbol\rho=(\rho_{s})_{s\ge0}$
with $\rho_{s}\in\mathcal D_{M}$ for all $s$, the value
$v[\boldsymbol\rho](t,x)$ defined by \eqref{definition:v-tau} is the unique
bounded classical solution of $($see \eqref{v-pde}$)$
\begin{equation}\label{eq:backward-relaxed}
\partial_{t}v+\bar{\mathcal L}_{\rho_{t}}v-\beta v+\bar f_{\rho_{t}}=0,
\qquad
\bar{\mathcal L}_{\rho_{t}}
=\tfrac12\operatorname{tr}(\bar a_{\rho_{t}}\nabla^{2})
+\bar b_{\rho_{t}}^{\top}\nabla,
\end{equation}
and satisfies the $t$-uniform bound
\begin{equation}\label{eq:v-bound-par}
\sup_{t\ge0}\big\|v[\boldsymbol\rho](t,\cdot)\big\|_{C^{2,\alpha'}_{b}}
\le C_{M}.
\end{equation}
Moreover, the value depends continuously on the policy in the weighted
topology: for flows $\boldsymbol\rho,\boldsymbol\rho'$ valued in
$\mathcal D_{M}$, writing
$\bar{\mathsf d}_{\kappa}
:=\sup_{s\ge0}\mathsf d_{\kappa}(\rho_{s},\rho'_{s})$,
\begin{equation}\label{eq:v-cont-weighted}
\sup_{t\ge0}\big\|\langle\cdot\rangle^{-\kappa}
\big(v[\boldsymbol\rho](t,\cdot)-v[\boldsymbol\rho'](t,\cdot)\big)\big\|_{\infty}
\le C_{M,\beta,\kappa}\,\bar{\mathsf d}_{\kappa},
\end{equation}
and, with $\lambda_{*}=\frac{2+\alpha}{2+\alpha'}\in(0,1)$, the locally
uniform $C^{2,\alpha}$ estimate
\begin{equation}\label{eq:v-cont-local}
\sup_{t\ge0}\ \sup_{x\in\mathbb R^{d}}\ \langle x\rangle^{-\kappa}
\big\|v[\boldsymbol\rho](t,\cdot)-v[\boldsymbol\rho'](t,\cdot)
\big\|_{C^{2,\alpha}(B_{1}(x))}
\le C_{M,\beta,\kappa}\,\bar{\mathsf d}_{\kappa}^{\,1-\lambda_{*}} .
\end{equation}
\end{lemma}
\begin{proof}
\emph{Proof of \eqref{eq:v-bound-par}.}
$v^\tau$ defined in \eqref{definition:v-tau} is the Feynman--Kac representation of
\eqref{eq:backward-relaxed} \cite[Theorem 5.7.6]{karatzas1991brownian};
boundedness of the running cost (Lemma~\ref{lem:entropy-relaxed}(a)) and
$\beta>0$ make it the unique bounded classical solution, by the maximum
principle for parabolic equations \cite[Chapter 2, Theorem 9]{friedman1964parabolic}.
By Assumption~\ref{assumption:parameter} and
Lemma~\ref{lem:entropy-relaxed}(a), \eqref{eq:backward-relaxed} is uniformly
parabolic with coefficients bounded in $C^{0,\alpha'}_{b}$ uniformly in $t$.
The global parabolic Schauder estimate
\cite[Theorem 8.11.1]{krylov1996lectures}, combined with the exponential
damping furnished by $-\beta$, gives
$\sup_{t}\|v(t,\cdot)\|_{C^{2,\alpha'}_{b}}\le C_{M}$ (see also
\cite[Chapter 5]{lunardi1995analytic}).

\emph{Proof of \eqref{eq:v-cont-weighted}.} The difference
$w:=v[\boldsymbol\rho]-v[\boldsymbol\rho']$ solves
\eqref{eq:backward-relaxed} (with policy $\boldsymbol\rho$) with the
additional source
\[
g_{t}=(\bar f_{\rho_{t}}-\bar f_{\rho'_{t}})
+(\bar{\mathcal L}_{\rho_{t}}-\bar{\mathcal L}_{\rho'_{t}})
v[\boldsymbol\rho'] .
\]
By the pointwise Lipschitz bound \eqref{eq:coeff-ptw-lip} of
Lemma~\ref{lem:entropy-relaxed}(b) and the uniform bound
\eqref{eq:v-bound-par} applied to $v[\boldsymbol\rho']$,
\begin{equation}\label{eq:source-ptw}
|g_{t}(x)|
\le C_{M}\big(1+\|v[\boldsymbol\rho'](t,\cdot)\|_{C^{2}_{b}}\big)
\|\rho_{t}(x,\cdot)-\rho'_{t}(x,\cdot)\|_{L^{1}_{2}}
\le C_{M}\,\langle x\rangle^{\kappa}\,\bar{\mathsf d}_{\kappa}
\qquad\text{for all }(t,x).
\end{equation}
Let $X^{t,x}$ denote the diffusion generated by
$\bar{\mathcal L}_{\rho_{\cdot}}$ started at $x$ at time $t$. Since
$\bar b_{\rho},\bar a_{\rho}$ are bounded uniformly over $\mathcal D_{M}$
(Lemma~\ref{lem:entropy-relaxed}(a)), the Burkholder--Davis--Gundy inequality
gives the polynomial moment estimate
\begin{equation}\label{eq:X-moment}
\mathbb E\big[\langle X^{t,x}_{t+s}\rangle^{\kappa}\big]
\le C_{\kappa}\big(\langle x\rangle^{\kappa}+1+s^{\kappa}\big),
\qquad s\ge0 .
\end{equation}
The Feynman--Kac representation of $w$,
$w(t,x)=\mathbb E\int_{0}^{\infty}e^{-\beta s}g_{t+s}(X^{t,x}_{t+s})\,ds$,
together with \eqref{eq:source-ptw}--\eqref{eq:X-moment}, yields
\[
|w(t,x)|
\le C_{M}\,\bar{\mathsf d}_{\kappa}
\int_{0}^{\infty}e^{-\beta s}\,
\mathbb E\big[\langle X^{t,x}_{t+s}\rangle^{\kappa}\big]\,ds
\le C_{M,\beta,\kappa}\,\langle x\rangle^{\kappa}\,\bar{\mathsf d}_{\kappa},
\]
since the polynomial growth in $s$ in \eqref{eq:X-moment} is integrable
against $e^{-\beta s}$ for \emph{every} $\beta>0$. This implies
\eqref{eq:v-cont-weighted}.

\emph{Proof of \eqref{eq:v-cont-local}.} Fix $x$. By the
interpolation inequality on balls
\cite[Theorem 3.2.1]{krylov1996lectures},
\[
\|w(t,\cdot)\|_{C^{2,\alpha}(B_{1}(x))}
\le C\,\|w(t,\cdot)\|_{C^{0}(B_{2}(x))}^{1-\lambda_{*}}
\|w(t,\cdot)\|_{C^{2,\alpha'}(B_{2}(x))}^{\lambda_{*}},
\qquad \lambda_{*}=\tfrac{2+\alpha}{2+\alpha'} .
\]
 On $B_{2}(x)$ one has $\langle y\rangle^{\kappa}\le C\langle x\rangle^{\kappa}$,  the pointwise form of \eqref{eq:v-cont-weighted},
$|w(t,y)|\le C_{M,\beta,\kappa}\langle y\rangle^{\kappa}\bar{\mathsf d}_{\kappa}$,
yields
$\|w(t,\cdot)\|_{C^{0}(B_{2}(x))}
\le C_{M,\beta,\kappa}\,\langle x\rangle^{\kappa}\,\bar{\mathsf d}_{\kappa}$, and
\eqref{eq:v-bound-par} bounds the second factor by $(2C_{M})^{\lambda_{*}}$.
Hence
$\|w(t,\cdot)\|_{C^{2,\alpha}(B_{1}(x))}
\le C\,\langle x\rangle^{\kappa(1-\lambda_{*})}
\bar{\mathsf d}_{\kappa}^{\,1-\lambda_{*}}$, and since
$\langle x\rangle^{-\kappa}\le\langle x\rangle^{-\kappa(1-\lambda_{*})}$,
\eqref{eq:v-cont-local} follows.
\end{proof}

We now state and prove the main result of this section. Since the value depends on the policy only H\"older-continuously at the level
of the H\"older norms (Lemma~\ref{lem:value}), the policy-update map cannot be
expected to be a contraction. We therefore obtain a fixed point from
Schauder's theorem.

\begin{proposition}[Global existence]\label{prop:existence-mkv}
Let Assumptions \ref{assumption:parameter}, \ref{assum:holder} hold,
$\beta,\lambda>0$, and $\mathrm{Law}(u_{0})\in\mathcal D_{M_{0}}$. Then the
coupled system \eqref{eq:DynU}--\eqref{definition:v-tau} admits a solution
$(v^{\tau},u^{x}_{\tau})_{\tau\ge0}$ with
$\rho_{\tau} = Law(u^x_\tau)\in\mathcal D_{M}$ and
$\sup_{\tau\ge0}\|v^{\tau}\|_{C^{2,\alpha'}_{b}}\le L<\infty$.
\end{proposition}

\begin{proof}
 The proof has five steps: we (1) define a convex set of
policy flows $\mathcal K_{T}$ and a map $\Phi$ on it whose fixed points are
exactly the solutions of the coupled system on $[0,T]$; (2) show
$\Phi(\mathcal K_{T})\subseteq\mathcal K_{T}$; (3) show $\mathcal K_{T}$ is
compact; (4) show $\Phi$ is continuous; (5) apply Schauder's theorem and
extend the solution to $[0,\infty)$ by concatenation.

\noindent
\emph{Step 1 (The fixed-point set and the map $\Phi$).}
Fix $T>0$ and let $M=M(T,M_{0})$ be the constant from
Lemma~\ref{lem:entropy-relaxed}(c). With a constant $C_{T}$ determined in
Step~2, define
\[
\mathcal K_{T}:=\Big\{\boldsymbol\rho=(\rho_{\tau})_{\tau\in[0,T]}
\in C([0,T];\mathcal D_{M}):\
\mathsf d_{\kappa}(\rho_{\tau},\rho_{\tau'})
\le C_{T}|\tau-\tau'|^{1/2}\ \ \forall\,\tau,\tau'\Big\},
\]
which is nonempty (it contains constant-in-$\tau$ flows) and convex, since
$\mathcal D_{M}$ is convex (Lemma~\ref{lem:DM-compact}) and the time modulus
is stable under convex combinations. Given $\boldsymbol\rho\in\mathcal K_{T}$,
define $\Phi(\boldsymbol\rho)$ in two stages:

\noindent \emph{Value stage.} Define $\bar{\boldsymbol\rho} = \{\bar\rho_t\}_{t\geq 0}$ by 
\begin{equation}
\bar\rho_t = \begin{cases}
\rho_{T-t}, &0\leq t \leq T;\\
\rho_0, & t > T.
\end{cases}
\end{equation}
Let
$v^{\tau}:=v[\bar{\boldsymbol\rho}](0,\cdot)$. By Lemma~\ref{lem:value},
$\sup_{\tau\le T}\|v^{\tau}\|_{C^{2,\alpha'}_{b}}\le C_{M}$, and by
\eqref{eq:v-cont-local} the map
$\boldsymbol\rho\mapsto(\nabla v^{\tau},\nabla^{2}v^{\tau})$ is continuous
from $(\mathcal K_{T},\sup_{\tau}\mathsf d_{\kappa})$ into
$C([0,T];C^{1,\alpha}(B_{R}))$ for every $R>0$.

\noindent\emph{Policy stage.} With this $v^{\tau}$, the drift
$G(\tau,x,u)=\nabla_{u}\mathbf H\big(x,u,\nabla v^{\tau}(x),
\nabla^{2}v^{\tau}(x)\big)$ in \eqref{eq:DynU-G} (which is same as \eqref{eq:DynU}) is Lipschitz with linear
growth in $u$ and bounded in $C^{0,\alpha'}_{b}$ in $x$, uniformly in
$\tau\le T$, by the bound in (a). Solve \eqref{eq:DynU} and denote
$\Phi(\boldsymbol\rho)_{\tau}(x,\cdot):=\mathrm{Law}(u^{x}_{\tau})$.

By construction, $\boldsymbol\rho^{\star}=\Phi(\boldsymbol\rho^{\star})$ if
and only if $(v^{\tau},(u^{x}_{\tau})_{x})_{\tau\le T}$ solves the coupled
system \eqref{eq:DynU}--\eqref{definition:v-tau} on $[0,T]$.

\noindent
\emph{Step 2 (Invariance: $\Phi(\mathcal K_{T})\subseteq\mathcal K_{T}$).}
$\Phi(\boldsymbol\rho)_{\tau}\in\mathcal D_{M}$ for all
$\tau\le T$ is exactly Lemma~\ref{lem:entropy-relaxed}(c), since the drift
$G$ satisfies its hypotheses by Step~1(b). For the time modulus, write
$\tilde\rho_{\tau}:=\Phi(\boldsymbol\rho)_{\tau}$ and recall from
Lemma~\ref{lem:entropy-relaxed}(c) that, for each $x$,
$\tilde\rho_{\tau}(x,\cdot)$ solves the Fokker--Planck equation
$\partial_{\tau}\tilde\rho=\lambda\Delta_{u}\tilde\rho
-\mathrm{div}_{u}(G\tilde\rho)$. Integrating in time and taking the
$L^{1}_{2}$-norm, for each fixed $x$,
\[
\|\tilde\rho_{\tau}(x,\cdot)-\tilde\rho_{\tau'}(x,\cdot)\|_{L^{1}_{2}}
\le\int_{\tau'}^{\tau}
\big\|\lambda\Delta_{u}\tilde\rho_{s}
-\mathrm{div}_{u}(G\tilde\rho_{s})\big\|_{L^{1}_{2}}\,ds
\le C_{T}'\,|\tau-\tau'|^{1/2},
\]
since $\mathrm{div}_{u}G$ is bounded and $|G|\le C\langle u\rangle$,
the moment bound (i) gives
$\|(\mathrm{div}_{u}G)\tilde\rho_{s}\|_{L^{1}_{2}}\le C_{T}$. On the other hand, since the diffusion coefficient is a constant, the parabolic estimate for
$\|\Delta_{u}\tilde\rho^{s}\|_{L^{1}_{2}}$ degenerates like $s^{-1/2}$ near
$s=0$. Integrating it directly leads to an upper bound which is multiple of $|\tau-\tau'|^{1/2}$. 
Since the weight satisfies $\langle x\rangle^{-\kappa}\le1$, and
$|\tau-\tau'|\le T^{1/2}|\tau-\tau'|^{1/2}$ on $[0,T]$, taking $\sup_{x}$
gives
$\mathsf d_{\kappa}(\tilde\rho^{\tau},\tilde\rho^{\tau'})
\le C_{T}|\tau-\tau'|^{1/2}$ with $C_{T}:=T^{1/2}C_{T}'$.  Hence
$\Phi(\boldsymbol\rho)\in\mathcal K_{T}$.

\noindent
\emph{Step 3 (Compactness of $\mathcal K_{T}$).}
By Lemma~\ref{lem:DM-compact}, $(\mathcal D_{M},\mathsf d_{\kappa})$ is
compact, and $\mathcal K_{T}\subset C([0,T];\mathcal D_{M})$ is
equicontinuous in $\tau$ by the uniform $\tfrac12$-H\"older time modulus.
The Arzel\`a--Ascoli theorem therefore shows that $\mathcal K_{T}$ is compact
in the metric $\sup_{\tau\le T}\mathsf d_{\kappa}$, and closedness of the modulus
condition under this convergence is immediate.

\noindent
\emph{Step 4 (Continuity of $\Phi$).}
Let $\boldsymbol\rho^{k}\to\boldsymbol\rho \in \mathcal K_T$ in
$\sup_{\tau}\mathsf d_{\kappa}$. By \eqref{eq:v-cont-local},
$\nabla v[\boldsymbol\rho^{k}],\nabla^{2}v[\boldsymbol\rho^{k}]
\to\nabla v[\boldsymbol\rho],\nabla^{2}v[\boldsymbol\rho]$ uniformly on
$[0,T]\times B_{R}$ for every $R>0$. Hence, for each fixed $x$, the frozen-$x$
drifts $G^{k}(\cdot,x,\cdot)\to G(\cdot,x,\cdot)$ locally uniformly, with
Lipschitz and growth constants uniform in $(k,x)$. Standard stability of SDEs
with respect to convergence of coefficients then gives
\[
\sup_{\tau\le T}\ \mathcal W_{2}\big(\Phi(\boldsymbol\rho^{k})^{\tau}(x,\cdot),
\Phi(\boldsymbol\rho)^{\tau}(x,\cdot)\big)\longrightarrow0,
\qquad\text{uniformly for }x\in B_{R},\ \forall R>0 .
\]
Moreover $\mathcal W_{2}$ is bounded on $\mathcal D_{M}$: by the moment bound
(i), $\mathcal W_{2}(\mu,\nu)^{2}\le2\int|u|^{2}(\mu+\nu)\le4M$. Splitting
$\sup_{x}=\sup_{|x|\le R}+\sup_{|x|>R}$ and using
$\langle x\rangle^{-\kappa}\le\langle R\rangle^{-\kappa}$ on the tail, we
conclude
$\sup_{\tau}\mathcal W_{2,\kappa}\big(\Phi(\boldsymbol\rho^{k})^{\tau},
\Phi(\boldsymbol\rho)^{\tau}\big)\to0$. Since all laws involved lie in
$\mathcal D_{M}$ by Step~2, the H\"older equivalence of
$\mathcal W_{2,\kappa}$ and $\mathsf d_{\kappa}$ on $\mathcal D_{M}$
(Lemma~\ref{lem:metric-equiv}) yields
$\Phi(\boldsymbol\rho^{k})\to\Phi(\boldsymbol\rho)$ in
$\sup_{\tau}\mathsf d_{\kappa}$, i.e.\ $\Phi$ is continuous on
$\mathcal K_{T}$.

\noindent
\emph{Step 5 (Fixed point and global extension).}
By Steps 2--4, $\Phi$ is a continuous self-map of the nonempty convex and compact
set $\mathcal K_{T}$, so Schauder's fixed point theorem yields
$\boldsymbol\rho^{\star}=\Phi(\boldsymbol\rho^{\star})$, i.e.\ a solution of
the coupled system on $[0,T]$, with $\rho^*_{\tau}\in\mathcal D_{M}$ and
$\sup_{\tau\le T}\|v^{\tau}\|_{C^{2,\alpha'}_{b}}\le L:=C_{M}$ by
Lemma~\ref{lem:value}. Since the constant $L$ does not depend on $T$,
we may iterate: for each $k\ge 1$, the terminal law $\mathrm{Law}(u^{x}_{kT})\in\mathcal D_{M}$
serves as initial datum for the next window $[kT,(k+1)T]$, and the solutions
concatenate into a solution on $[0,\infty)$.
\end{proof}
}

\section{Convergence analysis}\label{sec: convergence}
In this section, we discuss the convergence rate of the continuous policy-value iteration dynamics \eqref{eq:DynU} \& \eqref{v-dynamics} under monotonicity conditions on the Hamiltonians. 
\begin{assumption}[Monotonicity Condition]\label{assum: monotone relax}
For each $x\in\mathbb R^d$, and for all $\tau\ge 0$, we assume (omitting the argument $x$ for simplicity)
\begin{align*}
(\text{MC I}): &\mathbb E\Big[\Big((\mathbf{H}^\tau- \lambda \ln \pi_\tau)(u_\tau)-(\tilde{\mathbf{H}}^\tau -\lambda\ln\tilde \pi_\tau)(\tilde u_\tau)\Big)\Big] (v^\tau-\tilde v^\tau)\le 0;\\
 (\text{MC II}): &\Big( \nabla_x \mathbb E\Big[\Big((\mathbf{H}^\tau- \lambda \ln \pi_\tau )(u_\tau)-(\tilde{\mathbf{H}}^\tau -\lambda\ln\tilde \pi_\tau)(\tilde u_\tau)\Big)\Big]\Big)^\top(v^{\tau}_x-\tilde v^{\tau}_x )\le 0;\\
 (\text{MC III}): &\tr\Big( \nabla^2_{xx}\mathbb E\Big[(\mathbf{H}^\tau- \lambda \ln \pi_\tau )(u_\tau)-(\tilde{\mathbf{H}}^\tau -\lambda\ln\tilde \pi_\tau)(\tilde u_\tau)\Big]:\\
 &(v^{\tau}_{xx}-\tilde v^{\tau}_{xx})\Big) \le 0,
\end{align*}
where $(\pi_\tau,v^\tau,\mathbf{H}^\tau)$ and $(\tilde\pi_\tau,\tilde v^\tau,\tilde{\mathbf{H}}^\tau)$ are the result of policy and value iteration, and the corresponding Hamiltonian (c.f. \eqref{H_tau}) at each $\tau$, i.e.
\[
\tilde{\mathbf{H}}^\tau:= \mathbf{H}(\tau, x, u, \tilde v^{\tau}_x, \tilde v^{\tau}_{xx}) = f(x,u) + b^\top(x,u)\tilde v^{\tau}_x(x)+ \frac{1}{2}\tr [\sigma \sigma^{\top}(x, u):\tilde v^{\tau}_{xx}(x)],
\]
starting from initial value $(u_0,v^0)$ and $(\tilde u_0,\tilde v^0)$, following the dynamics defined in Lemma \ref{dynamic v relax}, respectively.
\end{assumption}

\begin{lemma}\label{convergence value relax}
Under Assumptions \ref{assumption:parameter} and \ref{assum: monotone relax}, the dynamics in Lemma \ref{dynamic v relax} satisfy the following estimates,
\begin{align*}
\|v^{\tau}(x)-\tilde v^{\tau}(x)\|^2&\le e^{-2\beta \tau} \|v^0(x)-\tilde v^0(x)\|^2;\\
\|v^{\tau}_x(x)-\tilde v^{\tau}_x(x)\|^2&\le e^{-2\beta \tau} \|v^0_x(x)-\tilde v^0_x(x)\|^2;\\
\|v^{\tau}_{xx}(x)-\tilde v^{\tau}_{xx}(x)\|^2&\le e^{-2\beta \tau} \|v^0_{xx}(x)-\tilde v^0_{xx}(x)\|^2.
\end{align*}
\end{lemma}
\begin{proof}
We show the convergence for $v^{\tau}_x-\tilde v^{\tau}_x$, and the estimates for $\|v^{\tau}(x)-\tilde v^{\tau}(x)\|^2$ and $\|v_{xx}^{\tau}(x)-\tilde v_{xx}^{\tau}(x)\|^2$ follow similarly. The difference between $v^{\tau}_x$ and $\tilde v^{\tau}_x$ follows the dynamics below (omitting the argument $x$ for simplicity),
\begin{align*}
&d(v^{\tau}_x-\tilde v^{\tau}_x)= -\beta(v^\tau_x-\tilde v_x^\tau) d\tau + \nabla_x \Big[\int_{\mathbb R^n}(\mathbf{H}^\tau- \lambda \ln \pi_\tau)\pi_\tau(\cdot,u)\,du\\
&\qquad -\int_{\mathbb R^n}(\tilde{\mathbf{H}}^\tau -\lambda\ln\tilde \pi_\tau)\tilde \pi_\tau(\cdot,\tilde u)\,d\tilde u \Big]d\tau.
\end{align*}
Thus,
\begin{equation*}
\begin{split}
&\frac{d}{d\tau} \|v^{\tau}_x-\tilde v^{\tau}_x\|^2
=-2\beta \|v^{\tau}_x-\tilde v^{\tau}_x\|^2 + 2\Big(\nabla_x\Big[
\int_{\mathbb R^n} (\mathbf{H}^\tau- \lambda \ln\pi_\tau)\pi_\tau(x,u)\,du\\
& -\int_{\mathbb R^n}(\tilde{\mathbf{H}}^\tau -\lambda\ln\tilde \pi_\tau)\tilde \pi_\tau(x,\tilde u)\,d\tilde u \Big]\Big)^\top(v^{\tau}_x-\tilde v^{\tau}_x)
\le  -2\beta\|v^{\tau}_x-\tilde v^{\tau}_x\|^2,
\end{split}
\end{equation*}
where the last inequality follows from Assumption \ref{assum: monotone relax}. This then implies
\begin{equation}\label{convergence: v'}
\|v^{\tau}_x(x)-\tilde v^{\tau}_x(x)\|^2\le e^{-2\beta \tau} \|v^0_x(x)-\tilde v^0_x(x)\|^2.
\end{equation}
\end{proof}
\begin{proposition}\label{prop:exp-convergence}
Assume (MC II) and (MC III) in Assumption \ref{assum: monotone relax} hold, and for every $x\in\mathbb R^d$, there exist positive constants $\overline{\varepsilon}(x)$ and $\kappa(x)\neq \beta$, such that
\begin{align}\label{eigen assumption}
    &\nabla^2_{uu}f(x,\cdot)  +L\sum_{i=1}^n\nabla^2_{uu}b_i(x,\cdot)+ L^2\sum_{j,k=1}^n\nabla^2_{uu}\sigma_{jk}(x,\cdot)\nonumber \\
    &\qquad+L\sum_{i,j,k=1}^n \nabla_u\sigma_{ik}(x,\cdot) (\nabla_{u}\sigma_{jk} (x,\cdot))^{\top}+\overline{\varepsilon}(x)(n+n^2)\mathbf{I}_n\preceq -\kappa(x)\mathbf{I}_n,
\end{align}
where $L$ is the $\tau$-uniform bound for $v^\tau_x$, $v^\tau_{xx}$, $\nabla_u b$, $\sigma$ and $\nabla_u \sigma$, i.e.,
\begin{equation}\label{condition: uniform bound}
\Big\{ \|v^{\tau}_x(x)\|^2, \|v^{\tau}_{xx}(x)\|^2_F, \sum_{i=1}^n\|\nabla_u b_i\|^2,  \|\sigma\|_F^2, \sum_{i,j=1}^n\|\nabla_u \sigma_{ij}\|^2\Big\}\le L,
\end{equation}
for all $\tau\ge 0$, and $\|A\|_F^2:=\sum_{i,j=1}^n|A_{ij}|^2$ denotes the Frobenius norm. Then
\begin{equation*}
\mathcal W_2^2(\pi_\tau,\tilde \pi_\tau)\le e^{-2\kappa(x)\tau} \mathbb E[\|u_0-\tilde u_0\|^2]+\frac{e^{-2\beta\tau }-e^{-2\kappa(x)\tau}}{2(\kappa(x)-\beta)} C_0,
\end{equation*}
where for some $\underline{\varepsilon}(x)>0$,
\begin{align}\label{constant C0}
  C_0:=
 \frac{1}{\underline{\varepsilon}(x)}L\left(\mathbb E[\|v^{0}_{x}(x)-\tilde v^{0}_{x}(x)\|^2]+L\mathbb E[\|v^{0}_{xx}(x)-\tilde v^0_{xx}(x)\|^2]\right)
  \geq 0.
\end{align}
\end{proposition}
\begin{proof}
We adapt the idea of synchronous coupling techniques for Langevin dynamics
\cite{EGZ19,Dalalyan2020} to the inhomogeneous
Langevin dynamics governing $u_\tau$. Denote as $(v^{\tau}_x,u_{\tau})$ and $(\tilde v^{\tau}_x,\tilde u_{\tau})$ the coupled processes with different initial conditions, driven by the same Brownian motion $B$. Then $u_{\tau}-\tilde u_{\tau}$ follows
\begin{equation}
d(u_\tau-\tilde u_\tau)=[\nabla_u \mathbf{H}^\tau(u_\tau) -\nabla_{u} \tilde{\mathbf{H}}^\tau(\tilde u_\tau)]d\tau.
\end{equation}
Observe that (omitting the argument $x$ unless ambiguity arises)
\begin{align*}
&\nabla_u \mathbf{H}^\tau(u_\tau) -\nabla_{u} \tilde{\mathbf{H}}^\tau(\tilde u_\tau)\\
=& \nabla_u f (\cdot,u_\tau) -\nabla_{u} f(\cdot,\tilde u_\tau)+ \sum_{i=1}^n\nabla_u b_i(\cdot,u_\tau)v^{\tau}_{x,i}-\sum_{i=1}^n\nabla_{u} b_i(\cdot,\tilde u_\tau) \tilde v^{\tau}_{x,i}\\
+&\sum_{i,j,k=1}^n\sigma_{ik}(\cdot,u_\tau)\nabla_u\sigma_{jk}(\cdot,u_\tau)v^{\tau}_{xx,ij}-\sum_{i,j,k=1}^n\sigma_{ik}(\cdot,\tilde u_\tau)\nabla_{u}\sigma_{jk}(\cdot,\tilde u_\tau)\tilde v^{\tau}_{xx, ij}\\
=&\int_0^1\frac{d}{dr} \nabla_u f(\cdot,r u_\tau+(1-r)\tilde u_\tau ) dr\\
&+\sum_{i=1}^n\Big(v^{\tau}_{x,i}\int_0^1 \frac{d}{dr}\nabla_u b_i(\cdot,r u_\tau+(1-r)\tilde u_\tau ) dr +\nabla_{u}b_i(\cdot,\tilde u_\tau)[v^{\tau}_{x,i}-\tilde v^{\tau}_{x,i}]\Big)\\
&+\sum_{i,j,k=1}^n \Big(\sigma_{ik}(\cdot,u_\tau)v^{\tau}_{xx,ij}\int_0^1 \frac{d}{dr}\nabla_u \sigma_{jk}(\cdot,r u_\tau+(1-r)\tilde u_\tau ) dr\Big.\\
&\Big.+\sigma_{ik}(\cdot,u_\tau)\nabla_{u}\sigma_{jk}(x,\tilde u_\tau)[v^{\tau}_{xx,ij}-\tilde v^{\tau}_{xx,ij}]\Big) \\
&+\sum_{i,j,k=1}^n\tilde v^{\tau}_{xx,ij}\nabla_u\sigma_{jk}(x,\tilde u_\tau)\int_0^1 \frac{d}{dr}\sigma_{ik}(x,r u_\tau+(1-r)\tilde u_\tau ) dr \\
=&\int_0^1\langle u_\tau-\tilde u_\tau, \nabla^2_{uu} f(\cdot,r u_\tau+(1-r)\tilde u_\tau ) \rangle dr\\
+& \sum_{i=1}^nv^{\tau}_{x,i}\int_0^1 \langle u_\tau-\tilde u_\tau, \nabla^2_{uu} b_i(\cdot,r u_\tau+(1-r)\tilde u_\tau ) \rangle dr+\sum_{i=1}^n\nabla_{u}b_i(\cdot,\tilde u_\tau)[v^{\tau}_{x,i}-\tilde v^{\tau}_{x,i}]\\
+&\sum_{i,j,k=1}^n\sigma_{ik}(\cdot,u_\tau)v^{\tau}_{xx,ij}\int_0^1 \langle u_\tau-\tilde u_\tau, \nabla^2_{uu} \sigma_{jk}(\cdot,r u_\tau+(1-r)\tilde u_\tau ) \rangle dr\\
+&\sum_{i,j,k=1}^n\sigma_{ik}(\cdot,u_\tau)\nabla_{u}\sigma_{jk}(\cdot,\tilde u_\tau)[v^{\tau}_{xx,ij}-\tilde v^{\tau}_{xx,ij}] \\
+&\sum_{i,j,k=1}^n \tilde v^{\tau}_{xx,ij}\nabla_{u}\sigma_{jk}(\cdot,\tilde u_\tau)\int_0^1 \langle u_\tau-\tilde u_\tau, \nabla_{u} \sigma_{ik}(\cdot,r u_\tau+(1-r)\tilde u_\tau ) \rangle dr\\
=& \mathbf U^f_\tau  (u_\tau-\tilde u_\tau) +\sum_{i=1}^nv^{\tau}_{x,i} \mathbf U^{b_i}_\tau  (u_\tau-\tilde u_\tau)+\sum_{i=1}^n\nabla_{u}b_i(\cdot,\tilde u_\tau)[v^{\tau}_{x,i}-\tilde v^\tau_{x,i}]\\
+&\sum_{i,j,k=1}^n\mathbf U^{\sigma_{jk}}_\tau  (u_\tau-\tilde u_\tau)\sigma_{ik}(\cdot,u_\tau)v^{\tau}_{xx,ij}\\
+&\sum_{i,j,k=1}^n\sigma_{ik}(\cdot,u_\tau)\nabla_{u}\sigma_{jk}(\cdot,\tilde u_\tau)[v^{\tau}_{xx,ij}-\tilde v^{\tau}_{xx,ij}]\\
+&\sum_{i,j,k=1}^n(\hat{\mathbf U}^{\sigma_{ik}}_\tau)^\top(u_\tau-\tilde u_\tau) \tilde v^{\tau}_{xx,ij}\nabla_{u}\sigma_{jk}(\cdot,\tilde u_\tau),
\end{align*}
where we denote
$\mathbf U^f_\tau := \int_0^1 \nabla^2_{uu} f(x, r u_\tau+(1-r)\tilde u_\tau ) dr$, $\mathbf U^{b_i}_\tau := \int_0^1 \nabla^2_{uu} b_i(x,r u_\tau+(1-r)\tilde u_\tau ) dr$ for $i=1,\cdots,n$, $\mathbf U^{\sigma_{ij}}_\tau := \int_0^1 \nabla^2_{uu} \sigma_{ij}(x,r u_\tau+(1-r)\tilde u_\tau ) dr$, and $\hat{\mathbf U}^{\sigma_{ij}}_\tau := \int_0^1 \nabla_{u} \sigma_{ij}(x,r u_\tau+(1-r)\tilde u_\tau ) dr$ for $i,j=1,\cdots,n$.
Plugging the above terms into the difference equation for the control process, we have
\begin{align*}
&\frac{d}{d\tau} (u_\tau-\tilde u_\tau)
=\mathbf U^f_\tau  (u_\tau-\tilde u_\tau) +\sum_{i=1}^nv^{\tau}_{x,i} \mathbf U^{b_i}_\tau  (u_\tau-\tilde u_\tau)\\
&+\sum_{i=1}^n\nabla_{ u}b_i(\cdot,\tilde u_\tau)[v^{\tau}_{x,i}-\tilde v^{\tau}_{x,i}]+\sum_{i,j,k=1}^n\mathbf U^{\sigma_{jk}}_\tau  (u_\tau-\tilde u_\tau)\sigma_{ik}(\cdot,u_\tau)v^{\tau}_{xx,ij}\\
&+\sum_{i,j,k=1}^n\sigma_{ik}(\cdot,u_\tau)\nabla_{ u}\sigma_{jk}(\cdot,\tilde u_\tau)[v^{\tau}_{xx,ij}-\tilde v^{\tau}_{xx,ij}]\\
&+\sum_{i,j,k=1}^n(\hat{\mathbf U}^{\sigma_{ik}}_\tau)^\top(u_\tau-\tilde u_\tau)\tilde v^{\tau}_{xx,ij}\nabla_{u}\sigma_{jk}(x,\tilde u_\tau)\\
&=\Big[\mathbf U^f_\tau  +\sum_{i=1}^nv^{\tau}_{x,i}\mathbf U^{b_i}_\tau+ \sum_{i,j,k=1}^n\mathbf U^{\sigma_{jk}}_\tau  \sigma_{ik}(\cdot,u_\tau)v^{\tau}_{xx,ij}\\
&+\sum_{i,j,k=1}^nv^{\tau}_{xx,ij}\hat{\mathbf U}^{\sigma_{ik}}_\tau(\nabla_{ u}\sigma_{jk}(\cdot,\tilde u_\tau))^\top\Big] (u_\tau-\tilde u_\tau)+\sum_{i=1}^n\nabla_{u}b_i(\cdot,\tilde u_\tau)[v^{\tau}_{x,i}-\tilde v^{\tau}_{x,i}]\\
&+\sum_{i,j,k=1}^n\sigma_{ik}(\cdot,u_\tau)\nabla_{u}\sigma_{jk}(\cdot,\tilde u_\tau)[v^{\tau}_{xx,ij}-\tilde v^{\tau}_{xx,ij}].
\end{align*}
Applying the above equation and Young's inequality ($2ab\le \varepsilon a^2+\frac{1}{\varepsilon}b^2 $, for $\varepsilon>0$),
\begin{align}
&\frac{d}{d\tau} \|u_\tau-\tilde u_\tau\|^2\nonumber
=2(u_\tau-\tilde u_\tau)^\top\Big[\mathbf U^f_\tau  +\sum_{i=1}^nv^{\tau}_{x,i} \mathbf U^{b_i}_\tau+ \sum_{i,j,k=1}^n\mathbf U^{\sigma_{jk}}_\tau \sigma_{ik}(\cdot,u_\tau)v^{\tau}_{xx,ij}\nonumber \\
& +\sum_{i,j,k=1}^n\hat{\mathbf U}^{\sigma_{ik}}_\tau( \nabla_{u}\sigma_{jk}(\cdot,\tilde u_\tau))^\top v^{\tau}_{xx,ij}\Big] (u_\tau-\tilde u_\tau) \nonumber\\
&+2\sum_{i=1}^n(\nabla_{u}b_i(x,\tilde u_\tau))^\top[u_\tau-\tilde u_\tau][v^{\tau}_{x,i}-\tilde v^{\tau}_{x,i}]\nonumber\\
&+2\sum_{i,j,k=1}^n\sigma_{ik}(x,u_\tau)[v^{\tau}_{xx,ij}-\tilde v^{\tau}_{xx,ij}](\nabla_{u}\sigma_{jk}(\cdot,\tilde u_\tau))^\top[u_\tau-\tilde u_\tau]\nonumber\\
\le &2(u_\tau-\tilde u_\tau)^\top\Big[\mathbf U^f_\tau  +\sum_{i=1}^nv^{\tau}_{x,i} \mathbf U^{b_i}_\tau+ \sum_{i,j,k=1}^n\mathbf U^{\sigma_{jk}}_\tau \sigma_{ik}(\cdot,u_\tau)v^{\tau}_{xx,ij}\nonumber \\
& +\sum_{i,j,k=1}^n\hat{\mathbf U}^{\sigma_{ik}}_\tau( \nabla_{u}\sigma_{jk}(x,\tilde u_\tau))^\top v^{\tau}_{xx,ij}\Big] (u_\tau-\tilde u_\tau) \nonumber\\
&+(\sum_{i=1}^n\varepsilon_{1i}+\sum_{i,j=1}^n\varepsilon_{2ij})\|u_\tau-\tilde u_\tau\|^2+ \sum_{i=1}^n\frac{1}{\varepsilon_{1i}}\| \nabla_{u}b_i(x,\tilde u_\tau)\|^2|v^{\tau}_{x,i}-\tilde v^{\tau}_{x,i}|^2\nonumber\\
&+\sum_{i,j=1}^n\frac{1}{\varepsilon_{2ij}}|\sigma_{ik}(\cdot,u_\tau)|^2|v^{\tau}_{xx,ij}-\tilde v^{\tau}_{xx,ij}|^2\|\nabla_{u}\sigma_{jk}(\cdot,\tilde u_\tau)\|^2,\nonumber
\end{align}
where $\varepsilon_{1i}$'s and $\varepsilon_{2ij}$'s are positive constants, satisfying $\max_{1\leq i,j\leq n}\left\{\varepsilon_{1i},\varepsilon_{2ij}\right\}=\overline{\varepsilon}(x)$.
Notice that by Assumption \ref{assumption:parameter}, $\nabla_u b$, $\sigma$ and $\nabla_u \sigma$ are bounded. Furthermore, Lemma \ref{convergence value relax} implies that $v^\tau_x$ and $v^\tau_{xx}$ converge, and thus are also bounded. Therefore there exists $L>0$, such that \eqref{condition: uniform bound} holds.
Applying Lemma \ref{convergence value relax} and using the upper bound $L$, the above inequality implies that
\begin{align*}
& \frac{d}{d\tau} \|u_\tau-\tilde u_\tau\|^2\\
\le & 2(u_\tau-\tilde u_\tau)^\top\Big[\mathbf U^f_\tau  +\sum_{i=1}^nv^{\tau}_{x,i} \mathbf U^{b_i}_\tau+ \sum_{i,j,k=1}^n\mathbf U^{\sigma_{jk}}_\tau \sigma_{ik}(\cdot,u_\tau)v^{\tau}_{xx,ij}\nonumber \\
&\qquad\qquad +\sum_{i,j,k=1}^n\hat{\mathbf U}^{\sigma_{ik}}_\tau( \nabla_{u}\sigma_{jk}(\cdot,\tilde u_\tau))^\top v^{\tau}_{xx,ij}+\overline{\varepsilon}(x)(n+n^2)\mathbf{I}_n\Big] (u_\tau-\tilde u_\tau) \nonumber\\
&+\frac{1}{\underline{\varepsilon}(x)}L\left(\|v^{\tau}_{x}-\tilde v^{\tau}_{x}\|^2+L\|v^{\tau}_{xx}-\tilde v^{\tau}_{xx}\|^2\right)\\
\le &2(u_\tau-\tilde u_\tau)^\top\Big[\mathbf U^f_\tau  +\sum_{i=1}^nv^{\tau}_{x,i} \mathbf U^{b_i}_\tau+ \sum_{i,j,k=1}^n\mathbf U^{\sigma_{jk}}_\tau \sigma_{ik}(\cdot,u_\tau)v^{\tau}_{xx,ij}\nonumber \\
&\qquad\qquad +\sum_{i,j,k=1}^n\hat{\mathbf U}^{\sigma_{ik}}_\tau( \nabla_{u}\sigma_{jk}(\cdot,\tilde u_\tau))^\top v^{\tau}_{xx,ij}+\overline{\varepsilon}(x)(n+n^2)\mathbf{I}_n\Big] (u_\tau-\tilde u_\tau) \nonumber\\
&+\frac{1}{\underline{\varepsilon}(x)}Le^{-2\beta\tau}\left(\|v^{0}_{x}-\tilde v^{0}_{x}\|^2+L\|v^{0}_{xx}-\tilde v^0_{xx}\|^2\right),
\end{align*}
where $\underline{\varepsilon}(x) = \min\left\{\varepsilon_{1i},\varepsilon_{2ij},1\leq i,j\leq n\right\}$.
Then condition \eqref{eigen assumption} implies that
\begin{align*}
&\frac{d}{d\tau} \|u_\tau-\tilde u_\tau\|^2\\
\le& -2\kappa(x) \|u_\tau-\tilde u_\tau\|^2+\frac{1}{\underline{\varepsilon}(x)}Le^{-2\beta\tau}\big(\|v^{0}_{x}-\tilde v^{0}_{x}\|^2+L\|v^{0}_{xx}-\tilde v^0_{xx}\|^2\big).
\end{align*}
Applying Gronwall's inequality, the following estimate holds
\begin{align*}\label{grownall relax}
 &\|u_\tau-\tilde u_\tau\|^2
 \le e^{-2\kappa(x)\tau} \|u_0-\tilde u_0\|^2\\
 &+e^{-2\kappa(x)\tau}\Big[
 \frac{1}{\underline{\varepsilon}(x)}L\Big(\|v^{0}_{x}-\tilde v^{0}_{x}\|^2+L\|v^{0}_{xx}-\tilde v^0_{xx}\|^2\Big) \int_0^\tau e^{2(\kappa(x)-\beta)r}dr\Big]\\
 \le& e^{-2\kappa(x)\tau} \|u_0-\tilde u_0\|^2\\
 &+\frac{e^{-2\beta\tau}-e^{-2\kappa(x)\tau}}{2(\kappa(x)-\beta)}\frac{1}{\underline{\varepsilon}(x)}L\Big(\|v^{0}_{x}-\tilde v^{0}_{x}\|^2 +L\|v^{0}_{xx}-\tilde v^0_{xx}\|^2\Big)\\
 =& e^{-2\kappa(x)\tau} \|u_0-\tilde u_0\|^2+\frac{e^{-2\beta\tau }-e^{-2\kappa(x)\tau}}{2(\kappa(x)-\beta)} C_0.
 \end{align*}
Taking expectation on both sides, we conclude that
\begin{equation*}
\mathcal W_2^2(\pi_\tau,\tilde \pi_\tau)\le e^{-2\kappa(x)\tau} \mathbb E[\|u_0-\tilde u_0\|^2]+\frac{e^{-2\beta\tau }-e^{-2\kappa(x)\tau}}{2(\kappa(x)-\beta)} C_0.
\end{equation*}
\end{proof}

\begin{corollary}
If $\pi_0 = \pi^* = \frac{e^{\mathbf{H}(x,u,v^*_x,v^*_{xx})/\lambda}}{\int_{\mathbb R^n}e^{\mathbf{H}(x,u,v^*_x,v^*_{xx})/\lambda} du}$, then $\pi_{\tau} = \pi^*$. Together with Proposition \ref{prop:exp-convergence}, it implies that $\pi_{\tau}$ converges exponentially.
\end{corollary}
\begin{proof}
If $\pi_0 = \pi^*$, then $v^0$ is independent of $t$, and
\begin{equation*}
v^0 = \mathbb E\Big[\int_0^\infty e^{-\beta t}\left(\tilde f(X^*_t,\pi^*(X^*_t))- \lambda  \int_{\mathbb R^n}\pi^*\ln\pi^*(X^*_t,u)du\right)dt\Big],
\end{equation*}
where $X^*$ is the state process under the strategy $\pi_t = \pi^*$ for $t \geq 0$. Since Proposition \ref{prop:optimality} shows that $\pi^*$ is the optimal strategy and $v^*$ is the corresponding value, $v^0 = v^*$.
Furthermore, by the policy improvement result in Proposition \ref{prop:policy-improvement}, $v^\tau = v^*$ for $\tau \geq 0$, thus the diffusion \eqref{eq:u-dynamics} that $u_\tau$ follows becomes
\begin{equation*}
	du^x_\tau = \nabla_u \mathbf{H} (x,u^x_\tau,v^*_x,v^*_{xx})d\tau+ \sqrt{2\lambda}dB_\tau,
\end{equation*}
which has an invariant distribution $\pi^*= \frac{e^{\mathbf{H}(x,u,v^*_x,v^*_{xx})/\lambda}}{\int_{\mathbb R^n}e^{\mathbf{H}(x,u',v^*_x,v^*_{xx})/\lambda} du'}$. Since $\pi_0 = \pi^*$, it implies that $\pi_{\tau} = \pi_0$.
\end{proof}
\section{Examples}\label{sec: example}
Consider a relaxed control problem in a linear-quadratic setting with $d=n=1$, in which $b(x,u) = Ax + Bu$, $\sigma(x,u) =Cx$ and $f = - \left(\frac{M}{2}x^2+ \frac{N}{2}u^2 + Px + Qu\right)$, $M \geq 0$, $N >0$ and $A<0$. Notice the model parameters do not satisfy the boundedness Assumption \ref{assumption:parameter}. However, as we show below, under our policy-value iterations, as long as we start the policy iteration from a Gaussian distribution, $u_\tau$ stays Gaussian with bounded second moment and $v^\tau$ is always quadratic in $x$ for every $\tau$. Furthermore, the linear dynamics of state $X$ imply that it has finite first and second moments at any $t\geq 0$. Finally, we also verify below that the relative entropies of $\pi_\tau$ are uniformly bounded in $\tau$. These conditions help us verify Assumptions \ref{assumption-T} and \ref{assumption:h} and use the dominated convergence theorem in all the proofs in Section \ref{sec: relax PGA}, and all the results in that section, in particular the policy improvement and convergence results, still hold while we omit the proofs. Since \(f\) is strictly concave in \(u\), and both \(b\) and \(\sigma\) are linear in \(u\), the second-order derivatives of \(b\) and \(\sigma\) with respect to \(u\) vanish. Hence,
\[
\nabla^2_{uu} f(x,u)
+ L \nabla^2_{uu} b_i(x,u)
+ L^2 \nabla^2_{uu} \sigma(x,u)
+ L \big(\nabla_u \sigma(x,\cdot)\big)^2
= -N < 0.
\]
It follows that \eqref{eigen assumption} holds for \(2\bar\varepsilon < N\), with \(\bar\varepsilon\) chosen independently of \(x\). Thus in the following we focus on verifying the monotonicity conditions (MC II) and (MC III) proposed in Section \ref{sec: convergence}, which guarantee the exponential convergence of the policy iteration.
The coupled dynamics \eqref{eq:DynU} \& \eqref{v-dynamics} for this model are reduced to (omitting the superscript and argument $x$ for the ease of notation)
\begin{align}
\begin{cases}
du_\tau =&\left(-Nu_\tau+Bv^\tau_x-Q\right)d\tau + \sqrt{2\lambda}dB_\tau,\\
dv^\tau =&\Big(\frac{C^2v^\tau_{xx}-M}{2}x^2+ (Av^\tau_x-P)x +(Bv^\tau_x-Q)\mathbb E[u_\tau] -\frac{N}{2}\mathbb E[u^2_\tau] \Big.\nonumber\\
&\left.- \lambda\int_{\mathbb R}\pi_{\tau}\ln \pi_{\tau}(u)du -\beta v^{\tau} \right)d\tau,\\
dv^\tau_x =&\Big((C^2v^\tau_{xx}-M + Av^\tau_{xx})x + C^2v^\tau_{xxx}x^2/2+ Av^\tau_x-P \Big.\nonumber\\
&\left.+Bv^\tau_{xx}\mathbb E[u_\tau] + (Bv^\tau_x -Q)\partial_x \mathbb E[u_\tau]- N\partial_x\mathbb E[u^2_\tau]/2\right.\nonumber\\
&\left.- \lambda\partial_x\int_{\mathbb R}\pi_{\tau}\ln \pi_{\tau}(u)du -\beta v^{\tau}_x \right)d\tau,\\
dv^\tau_{xx} =&\Big((2C^2+A)v^\tau_{xxx} x +C^2v^\tau_{xx}-M + 2Av^\tau_{xx} + C^2v^\tau_{xxxx}x^2/2 \Big.\nonumber\\
&\left.+Bv^\tau_{xxx} \mathbb E[u_\tau] + 2Bv^\tau_{xx}\partial_x \mathbb E[u_\tau]+(Bv^\tau_x -Q)\partial_{xx} \mathbb E[u_\tau]\right.\nonumber\\
&\left.-N\partial_{xx}\mathbb E[u^2_\tau]/2- \lambda\partial_{xx}\int_{\mathbb R}\pi_{\tau}\ln \pi_{\tau}(u)du -\beta v^{\tau}_{xx} \right)d\tau.
\end{cases}
\end{align}
Thus $u$ is an Ornstein-Uhlenbeck process with inhomogeneous coefficients, and
\begin{equation}
u_\tau = e^{-N\tau}u_0 + \int_0^\tau e^{-N(\tau-s)}(Bv^s_x-Q)ds + \sqrt{2\lambda}\int_0^\tau e^{-(\tau-s)N}dB_s,\nonumber
\end{equation}
which follows a Gaussian distribution given $u_0$ is Gaussian, and
\begin{align}
\mathbb E[u_\tau] =& e^{-N\tau}\mathbb E[u_0] + e^{-N\tau}\int_0^\tau e^{Ns}(Bv^s_x-Q)ds\label{u-mean},\\
\mathsf{Var}(u_\tau) =& e^{-2N\tau}\mathsf{Var}(u_0) + \lambda(1-e^{-2N\tau})/N,\label{u-variance}\\
-\int_{\mathbb R}\pi_\tau(u)\ln\pi_\tau(u)du =& \ln\Big(2e\pi\Big(e^{-N\tau}\mathsf{Var}(u_0) + \frac{\lambda(1-e^{-2N\tau})}{N}\Big)\Big)/2.\label{entropy}
\end{align}
\begin{lemma}\label{lem:vxx-indep-x}
If $\mathbb E[u_0]$ is linear in $x$, $\mathsf{Var}(u_0)$ is independent of $x$, and $v^s$ is a quadratic function for every $0\leq s\leq \tau$, then $\frac{dv^\tau_{xx}}{d\tau}$ is independent of $x$.
\end{lemma}
\begin{proof}
Since $v^\tau$ is quadratic in $x$ for all $0\leq s\leq \tau$, $dv^\tau_{xx}$ reduces to
\begin{align}
dv^\tau_{xx} =&\left((C^2+2A)v^\tau_{xx}-M +2Bv^\tau_{xx}\partial_x \mathbb E[u_\tau]+(Bv^\tau_x -Q)\partial_{xx} \mathbb E[u_\tau]\right.\nonumber\\
&\Big. - \frac{N}{2}\partial_{xx}\mathbb E[u^2_\tau] - \lambda\partial_{xx}\int_{\mathbb R}\pi_{\tau}\ln \pi_{\tau}(u)du -\beta v^{\tau}_{xx} \Big)d\tau.\nonumber
\end{align}
Furthermore, since $v^s_{xx}$ is constant for $0\leq s\leq \tau$ and $\mathsf{Var}(u_0)$ is independent of $x$, \eqref{u-mean}, \eqref{u-variance} and \eqref{entropy} imply that $\mathbb E[u_\tau]$ is linear in $x$, and $\mathsf{Var}(u_\tau)$ and $\int_{\mathbb R}\pi_{\tau}\ln \pi_{\tau}(u)du$ are independent of $x$. Therefore, $\mathbb E[u_\tau^2]$ is quadratic in $x$ and
\begin{align}
\frac{dv^\tau_{xx}}{d\tau} =&(C^2+2A)v^\tau_{xx}-M +2Bv^\tau_{xx}\partial_x \mathbb E[u_\tau] - \frac{N}{2}\partial_{xx}\mathbb E[u^2_\tau] -\beta v^{\tau}_{xx},\nonumber
\end{align}
is independent of $x$.
\end{proof}
From the above lemma, if $u_0$ is Gaussian, with mean linear and variance independent of $x$ (which we expect to be satisfied by the optimal solution), then $v^0$ is quadratic in $x$ (cf. the proof of Theorem 4.2 in \cite{wang2020continuous}), and the same holds for all $\tau\geq 0$, and $\mathsf{Var}(u_\tau)$ is independent of $x$.  Then $v^\tau$ can be written as
\begin{align}\label{v tau decomposition}
v^\tau = a_2^\tau x^2/2 + a^\tau_1 x + a^\tau_0,\qquad v^\tau_x = a^\tau_2 x + a^\tau_1,\qquad v^\tau_{xx} = a^\tau_2.
\end{align}
Letting $\mathbb E[u_0] = \mu$ and $\mathsf{Var}(u_0) = \sigma^2$, then algebraic calculation shows that
\begin{align*}
\partial_x \mathbb E[u_\tau^2] =& \frac{\partial \mathbb E[u_\tau]^2}{\partial x} = 2\mathbb E[u_\tau]\partial_x \mathbb E[u_\tau],\quad \partial_{xx} \mathbb E[u_\tau^2]= 2(\partial_x \mathbb E[u_\tau])^2,\ \partial_x\mathbb E[u_\tau]= BI^\tau_2,\\
\mathbb E[u_\tau] =& BI^\tau_2x +e^{-N\tau}\mu +  BI^\tau_1 - (1-e^{-N\tau})Q/N,
\end{align*}
where $I^\tau_1 = \int_0^\tau e^{N(s-\tau)}a^s_1ds$ and $I^\tau_2 = \int_0^\tau e^{N(s-\tau)}a^s_2ds$. Thus
\begin{align}
&dv^\tau_{xx}/d\tau =da^\tau_2d\tau
=\left(C^2+2A + 2B^2I^\tau_2 -\beta\right)a^\tau_2-M- NB^2(I^\tau_2)^2.\label{vxx-LQ}
\end{align}
\begin{lemma}\label{lemma:v0}
With $u_0\sim \mathcal N(\mu,\sigma^2)$ and definition \eqref{v tau decomposition}, for $\tau=0$, we have,
\begin{align*}a^0_1 = -\frac{BMC^2\mu}{(A+C^2)(\beta-2A-C^2)(\beta - A)}-\frac{P}{\beta-A},\quad a^0_2 =-\frac{M}{\beta-2A-C^2}.
\end{align*}
\end{lemma}
\begin{proof}
With the strategy $\rho^0_t = \pi_0$ for every $t\geq 0$ and initial value $X_0 = x$,
\begin{align}
dX_t = \left(AX_t +B\mu\right)dt +CX_tdW_t.\nonumber
\end{align}
Then $X_t = xe^{(A-\frac{1}{2}C^2)t +CW_t} + B\mu\int_0^te^{(A-\frac{1}{2}C^2)(t-s) +C(W_t-W_s)}ds$, and
\begin{align}
v^0 =& \mathbb E\left[\int_0^\infty e^{-\beta t}\left(-\frac{M}{2}X_t^2 - \frac{N}{2}\left(\mu^2+\sigma^2\right) -PX_t-Q\mu\right)dt\right]\nonumber\\
=&-\frac{M}{\beta-2A-C^2}x^2-\left(\frac{BMC^2\mu}{(A+C^2)(\beta-2A-C^2)(\beta - A)}+\frac{P}{\beta-A}\right)x + K,\nonumber
\end{align}
for some constant $K$.
\end{proof}
\begin{lemma}
The limit $v^* = \frac{a^*_2}{2}x^2 + a^*_1 x + a^*_0$ satisfies:
$a^*_2  := \lim\limits_{\tau\rightarrow\infty} a^\tau_2  = -\frac{2M}{\beta-C^2-2A + \sqrt{(\beta-C^2-2A)^2+\frac{4MB^2}{N}}}$ and $a^*_1  = \lim\limits_{\tau\rightarrow\infty}  a^\tau_1  = \frac{BQa^*_2+NP}{N(A-\beta) + B^2a^*_2}$.
\end{lemma}
\begin{proof}
By \eqref{vxx-LQ}, $I_2^\tau$ and $a^\tau_2$ satisfy the following coupled equations
\begin{align}
\frac{da^\tau_2}{d\tau}=&\left(C^2+2A + 2B^2I^\tau_2 -\beta\right)a^\tau_2-M- NB^2(I^\tau_2)^2,\label{eq:dynamics-a2}\\
\frac{dI^\tau_2}{d\tau} =& -NI^\tau_2 + a^\tau_2\label{eq:dynamics-i2}.
\end{align}
Since $a^\tau_2=v^\tau_{xx}$ converges to $a^*_2$ and thus $\lim\limits_{\tau\rightarrow\infty} I^\tau_2=I^*_2 = a^*_2/N$. The limit can be found by checking the equilibrium point $(a^*_2,I^*_2)$ of the above equations, i.e., $\frac{da^\tau_2}{d\tau}=\frac{dI^\tau_2}{d\tau}=0$ at $\tau = \infty$,
\begin{align*}
\left(C^2+2A + 2B^2I^*_2 -\beta\right)a^*_2-M - NB^2(I^*_2)^2 =0,\quad 
-NI^*_2 + a^*_2 = 0.
\end{align*}
Since $\mathbb E[u_\tau]$ is linear in $x$, and thus $\tilde f(x,\pi_\tau)$ is concave in $x$, and so is $v^\tau$. Thus we choose the negative root of the quadratic function of $a^*_2$ above, and $a^*_2 = -\frac{2M}{\beta-C^2-2A + \sqrt{(\beta-C^2-2A)^2+\frac{4MB^2}{N}}}$.
Next, since $a^\tau_1$ converges to $a^*_1$, $I^\tau_1 = \int_0^\tau e^{N(s-\tau)}a^s_1ds$ converges to $I^*_1 = a^*_1/N$. Its dynamics is
\begin{align}
\frac{da^\tau_1}{d\tau}=\frac{dv^\tau_x}{d\tau} -\frac{da^\tau_2}{d\tau}x
=&\left(A-\beta\right)a^\tau_1-P+(Ba^\tau_1- Q)BI^\tau_2\nonumber\\
&+\Big(Ba^\tau_2- NBI^\tau_2\Big)\Big(e^{-N\tau}\mu +  BI^\tau_1 - \frac{1-e^{-N\tau}}{N}Q\Big),\nonumber
\end{align}
and the right hand side should converge to $0$ at $\tau = \infty$. It implies that
$
(A-\beta)a^*_1 - P + Ba^*_2(Ba^*_1-Q)/N = 0.
$ Thus, we have $a^*_1 = \frac{BQa^*_2+NP}{N(A-\beta) + B^2a^*_2}$.
\end{proof}
\begin{proposition}
For the policy-value iterations which start from a Gaussian distribution with mean and variance independent of $x$, if $N>|A|$ and both are sufficiently large, then (MC II) and (MC III) are satisfied.
\end{proposition}
\begin{proof}
Suppose the initial values $u\sim \mathcal N(\mu,\sigma^2)$ and $\tilde u\sim \mathcal N(\tilde\mu,\tilde\sigma^2)$. Lemma \ref{lemma:v0} implies that $a^0_2 = \tilde a^0_2$, and since both $a^\tau_2$ and $\tilde a^\tau_2$ satisfy the differential equations \eqref{eq:dynamics-a2} and \eqref{eq:dynamics-i2}, $a^\tau_2 = \tilde a^\tau_2$ and $I^\tau_2 = \tilde I^\tau_2$. Thus monotonicity condition (MC III) is satisfied.
Next, since $\mathbb E[\mathbf{H}^\tau(u_\tau)]$ is quadratic in $x$, $\partial_x\mathbb E[\mathbf{H}^\tau(u_\tau)] = \partial_{xx}\mathbb E[\mathbf{H}^\tau(u_\tau)]x + C_\tau$, where
\begin{align}
C_\tau = &-\left(Ne^{-N\tau}\mu +  BNI^\tau_1 +e^{-N\tau}Q\right)BI^\tau_2\nonumber\\
&+\left(A+B^2I^\tau_2\right)a^\tau_1 + B\Big(e^{-N\tau}\mu +  BI^\tau_1 - \frac{1-e^{-N\tau}}{N}Q\Big)a^\tau_2-P. \nonumber
\end{align}
and $a^\tau_1$ follows
\begin{align}
\frac{da^\tau_1}{d\tau}=\frac{dv^\tau_x}{d\tau} -\frac{da^\tau_2}{d\tau}x\nonumber
=&\left(A-\beta\right)a^\tau_1-P+(Ba^\tau_1- Q)BI^\tau_2\nonumber\\
&+B\left(a^\tau_2- NI^\tau_2\right)\Big(e^{-N\tau}\mu - \frac{1-e^{-N\tau}}{N}Q + BI^\tau_1\Big). \nonumber
\end{align}
Thus (notice that $\partial_{xx}\mathbb E[\mathbf{H}^\tau(u_\tau)] - \partial_{xx}\mathbb E[\tilde{\mathbf{H}}^\tau(\tilde u_\tau)]=0$)
\begin{align}
&\Big(\partial_{x}\mathbb E[\mathbf{H}^\tau(u_\tau)] - \partial_{x}\mathbb E[\tilde{\mathbf{H}}^\tau(\tilde u_\tau)]\Big)(v^\tau_{x}-\tilde v^\tau_{x}) \nonumber\\
=&\Big((\partial_{xx}\mathbb E[\mathbf{H}^\tau(u_\tau)] - \partial_{xx}\mathbb E[\tilde{\mathbf{H}}^\tau(\tilde u_\tau)])x + C_\tau-\tilde C_\tau\Big)((a^\tau_{2}-\tilde a^\tau_{2})x + a^\tau_1-\tilde a^\tau_1)\nonumber \\
=&(C_\tau-\tilde C_\tau)(a^\tau_{1}-\tilde a^\tau_{1})\\
=&e^{-N\tau}B(a^\tau_2-NI^\tau_2)(\mu-\tilde \mu)(a^\tau_1-\tilde a^\tau_1)\nonumber\\
&+ B^2\Big(a^\tau_2-NI^\tau_2\Big)\Big(I^\tau_1-\tilde I^\tau_1\Big)(a^\tau_1-\tilde a^\tau_1)+\left(A+ B^2 I^\tau_2\right)\left(a^\tau_1-\tilde a^\tau_1\right)^2.\label{eq:Ca}
\end{align}
Let $Y_{\tau} = \frac{I^\tau_1 - \tilde I^\tau_1}{a^\tau_1 - \tilde a^\tau_1}$ and $Z_{\tau} = e^{-N\tau}\frac{\mu - \tilde \mu}{a^\tau_1 - \tilde a^\tau_1}$. With $Y_{0} = 0$ and $Z_0 = -(A+C^2)(\beta-2A-C^2)(\beta - A)/(BMC^2)$,
\begin{align}
\frac{dY_{\tau}}{d\tau} =& -\left(N+A-\beta + B^2I^\tau_2\right)Y_\tau+1-B(a^\tau_2-NI^\tau_2)Z_\tau Y_{\tau}\nonumber\\
& - B^2(a^\tau_2-NI^\tau_2)Y^2_{\tau},\label{eq:Y-dynamics}\\
\frac{dZ_\tau}{d\tau} =& -(N+A-\beta+B^2I^\tau_2)Z_\tau - B^2(a^\tau_2 - NI^\tau_2)Z_\tau Y_\tau+B(a^\tau_2 - NI^\tau_2)Z^2_\tau.\label{eq:Z-dynamics}
\end{align}
Since $a^\tau_2$ converges to $a^*_2$ and $I^\tau_2$ converges to $I^*_2 = a^*_2/N$, $I^\tau_2$ is bounded, and the bound does not depend on initial distribution $\pi_0$. For $A<0$ and $|A|$ sufficiently large, $A +B^2 I^\tau_2<0$, and with sufficiently large $N>|A-\beta|$, $N+A-\beta+B^2I^\tau_2>0$. In \eqref{eq:Y-dynamics} and \eqref{eq:Z-dynamics}, the coefficients of $Z_\tau Y_\tau$, $Y^2_\tau$ and $Z_\tau^2$ are all bounded, and the bounds are independent of the initial distribution $\pi_0$ and $N$. Then with sufficiently large $M$ and $N$, such that $|Z_0|$ is sufficiently small, since $Y_0 = 0$, the mean reverting term in both \eqref{eq:Y-dynamics} and \eqref{eq:Z-dynamics} dominates other terms. Thus $Y_\tau$ and $Z_\tau$ are bounded. Furthermore, since $a^\tau_2 - NI^\tau_2$ converges to $0$, $\lim\limits_{\tau\rightarrow \infty}Y_\tau = \frac{1}{N+A-\beta + B^2a^*_2/N}$ and $\lim\limits_{\tau\rightarrow \infty}Z_\tau = 0$. Therefore, \eqref{eq:Ca} implies that
\begin{align}
&(C_\tau - \tilde C_\tau)(a^\tau_1-\tilde a^\tau_1) \nonumber\\
=&\left(e^{-N\tau}B(a^\tau_2-NI^\tau_2)Z_\tau+ B^2\left(a^\tau_2-NI^\tau_2\right)Y_\tau+ A+ B^2 I^\tau_2\right)\left(a^\tau_1-\tilde a^\tau_1\right)^2. \label{a tau difference}
\end{align}
Notice that Lemma \ref{lemma:v0} implies that $|a^0_2|$ is decreasing in $|A|$, and $Z_0$ can be made small by choosing large $M$ (e.g. $M = |A|^4$). For sufficiently large $|A|$, the coefficient of $(a^\tau_1 - \tilde a^\tau_1)^2$ in \eqref{a tau difference} is negative.
\end{proof}

We also report two experiments: the Linear-Quadratic (LQ) model described above with a non-concave reward, simulated through the moment dynamics \eqref{u-mean}--\eqref{entropy}, and a constructed non-LQ example with a skewed-Laplace optimal control, solved by a density-based implementation of the coupled dynamics.

\subsection{The LQ model with a non-concave reward}\label{sec:num-lq}
We take $A=-2$, $B=2$, $C=1$, $M=-3$, $N=6$, $P=Q=1$, $\beta=0.3$ and $\lambda=0.1$ in the above LQ model. Since $M<0$, the reward is no longer concave and the model falls outside the verification argument above. For Gaussian $u_0$ the control stays Gaussian along the iteration, so the coupled dynamics reduce to ordinary differential equations for $(v^\tau,v^\tau_x,v^\tau_{xx})$ coupled with the moment dynamics \eqref{u-mean}--\eqref{entropy}, initialized at $u_0\sim\mathcal N(1,1)$; the value iteration uses an Euler scheme with $\Delta\tau=0.05$ up to $\tau_{\mathrm{end}}=3.5$ on the state grid $x\in[0.1,5]$. The empirical control distributions are obtained from $10^4$ simulated control trajectories started from $u_0\sim\mathcal N(1,1)$ at the state $x=0.1$, and running with the same Euler scheme up to $\tau_{\mathrm{end}}=3.5$. As shown in \Cref{fig:lq-noncc}, the dynamics remain stable, the value iterates converge, and the terminal control distribution agrees closely with the Gibbs distribution of the terminal value iterate, so the iteration remains effective beyond the concave setting.

\begin{figure}[htbp]
  \centering
  \includegraphics[height=3.0cm]{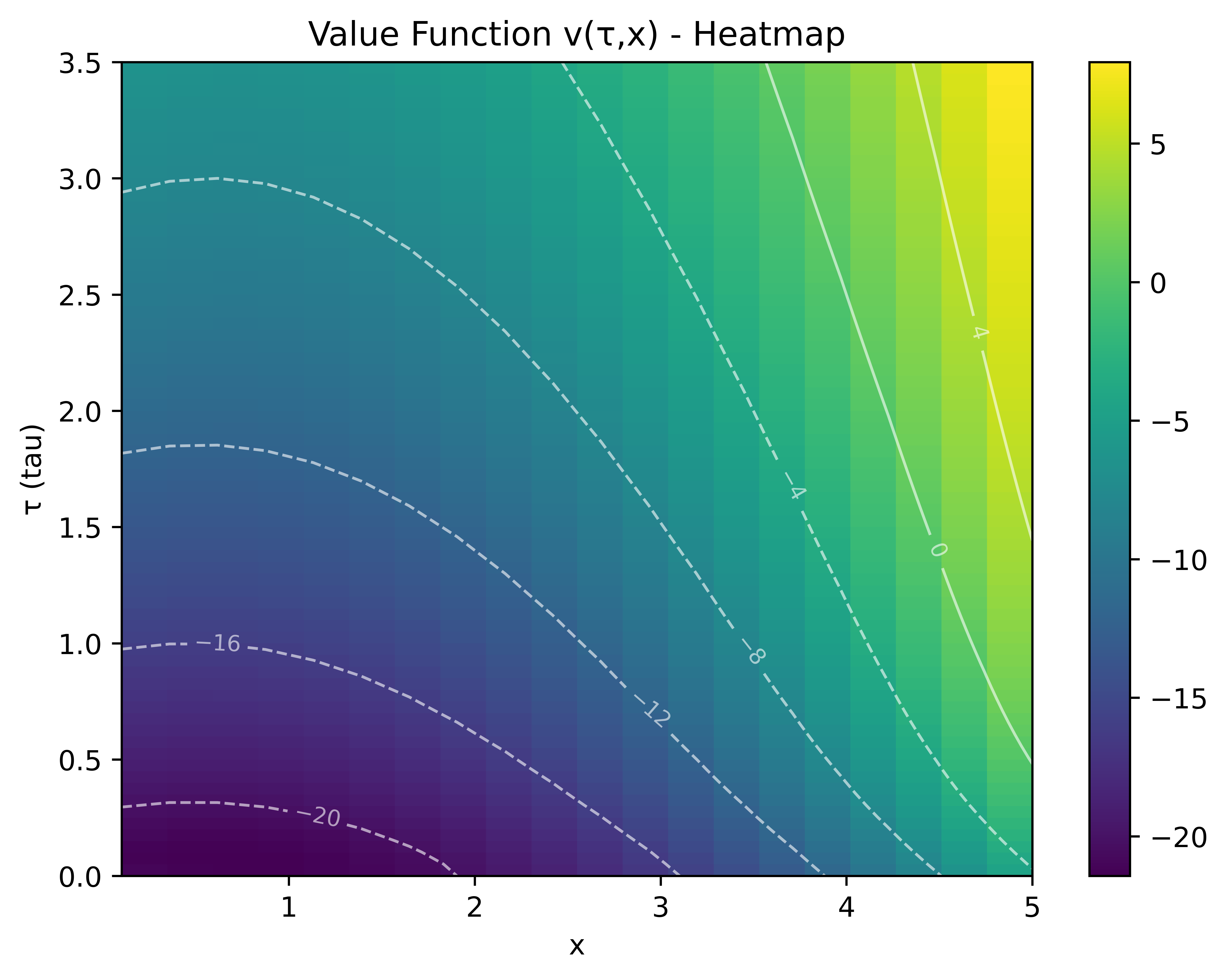}\hfill
  \includegraphics[height=3.0cm]{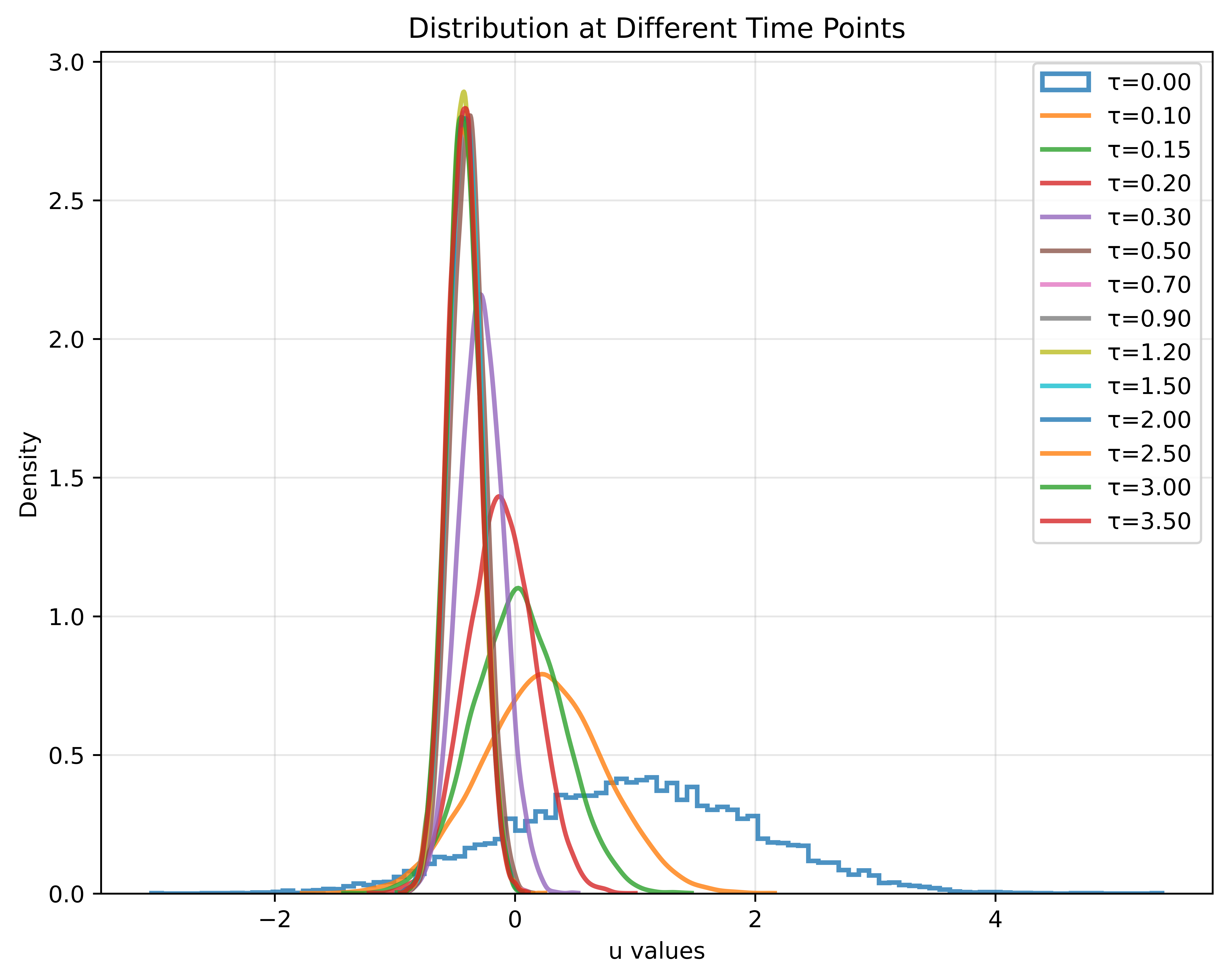}\hfill
  \includegraphics[height=3.0cm]{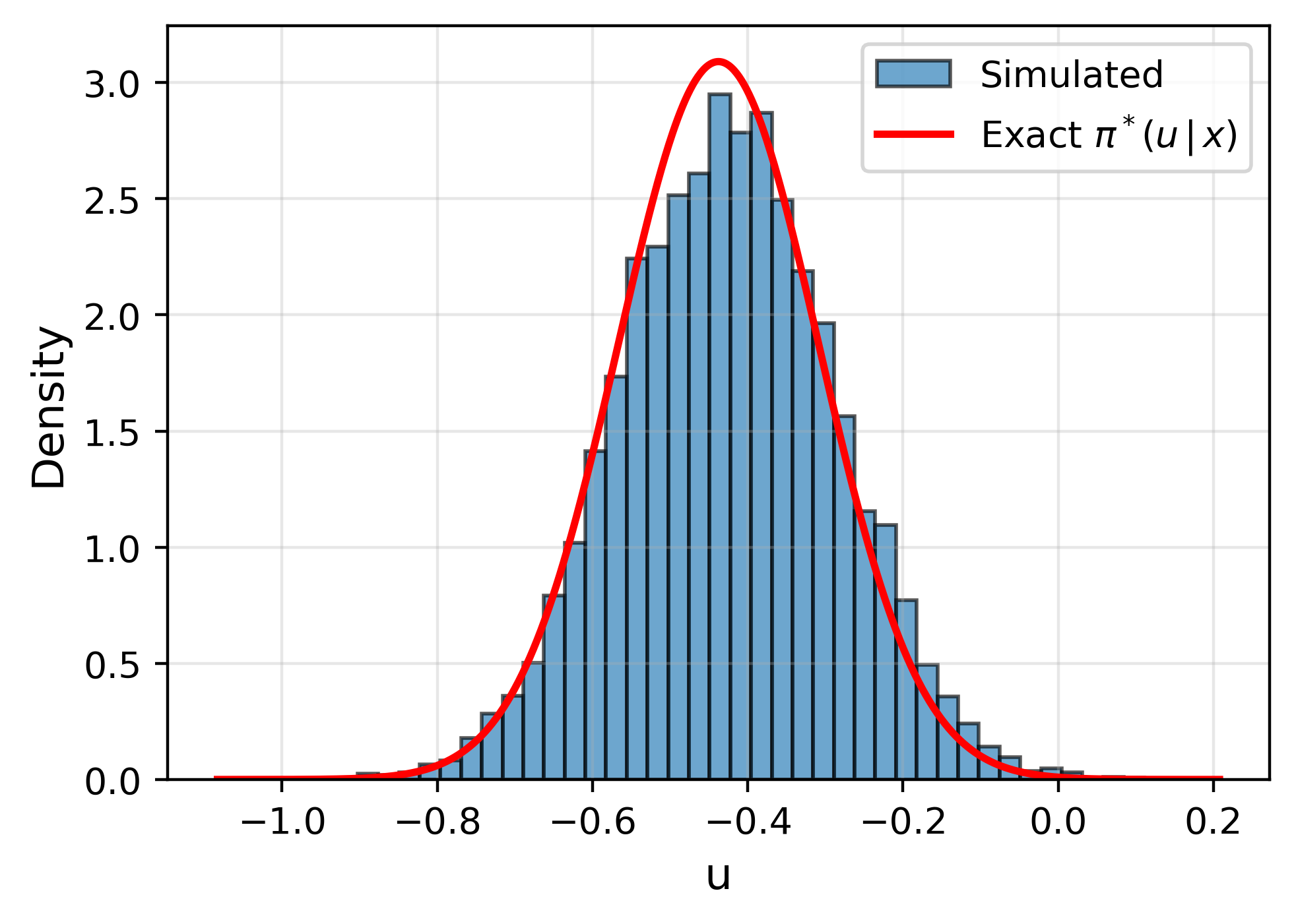}%
  \caption{From left to right: heatmap of the value iterates $v^\tau(x)$ over $(x,\tau)$; distribution of the simulated control at $x=0.1$ along the iteration; empirical distribution of the simulated control at the terminal iteration versus the exact optimal density $\pi^*(u\,|\,x)=\mathcal N\!\left((Bv^*_x(x)-Q)/N,\,\lambda/N\right)$, computed in closed form from the analytic solution.}
  \label{fig:lq-noncc}
\end{figure}

\subsection{Beyond LQ: a skewed-Laplace example}\label{sec:num-skewlap}
Without the Gaussian structure of the LQ model we implement \eqref{eq:DynU} \& \eqref{v-dynamics} at the level of the control density: for each state on a grid, $\pi_\tau(x,\cdot)$ is evolved by an explicit finite-difference scheme for the Fokker--Planck equation of the Langevin dynamics of $u^x_\tau$, while $(v^\tau,v^\tau_x,v^\tau_{xx})$ follow Lemma~\ref{dynamic v relax}, all integrals against $\pi_\tau$ being computed by quadrature and $v^0$ by Monte Carlo simulation under the initial policy ($400$ paths, horizon $T=60$). We apply this solver to a non-LQ example in $d=n=1$, designed so that the entropy-regularized HJB equation corresponds to a non-Gaussian optimal relaxed control. Choosing $v(x)=\log(1+x^2)$ and the control kernel $\phi(x,u)=a(x)u-\kappa(x)|u|$ with $a(x)=\tanh(x)$ and $\kappa(x)=2+\cos^2(x)$ gives the closed-form normalizer $Z(x)=\int_{\mathbb R}e^{\phi(x,u)/\lambda}du=2\kappa(x)\lambda/(\kappa^2(x)-a^2(x))$; setting $\mathbf H(x,u)=\beta v(x)+\phi(x,u)-\lambda\log Z(x)$ then yields $\lambda\log\int_{\mathbb R}e^{\mathbf H(x,u)/\lambda}du=\beta v(x)$, so that $v$ solves the HJB equation exactly and the running reward is recovered as $f=\mathbf H-b\,v_x-\frac12\sigma^2v_{xx}$. We take the mean-reverting $b(x,u)=-\theta x+\varepsilon_b\tanh(u)$ and $\sigma(x,u)=\sigma_0(1+\eta_x\tanh^2(x))(1+\eta_u\tanh^2(u))$ with $\theta=0.8$, $\varepsilon_b=0.2$, $\sigma_0=0.25$, $\eta_x=\eta_u=0.2$, $\beta=0.3$ and $\lambda=0.5$, so the optimal density $\pi^*(u\,|\,x)=Z(x)^{-1}\exp((a(x)u-\kappa(x)|u|)/\lambda)$ is a skewed Laplace distribution. Running the solver from a uniform initial density on $[-15,15]$, on a grid of $1000$ points over $[-50,50]$ and states $x\in[1,2]$, with step $\Delta\tau=10^{-3}$ (below the stability limit $\Delta u^2/(2\lambda)$ of the explicit scheme) up to $\tau_{\mathrm{end}}=40\gg1/\beta$, \Cref{fig:skewlap} shows the iteration concentrating from the flat initial density and recovering the skewed, non-Gaussian shape of $\pi^*$, which matches the exact optimal density at convergence: the coupled dynamics apply beyond LQ models and Gaussian optimal policies.

\begin{figure}[htbp]
  \centering
  \includegraphics[width=0.96\textwidth]{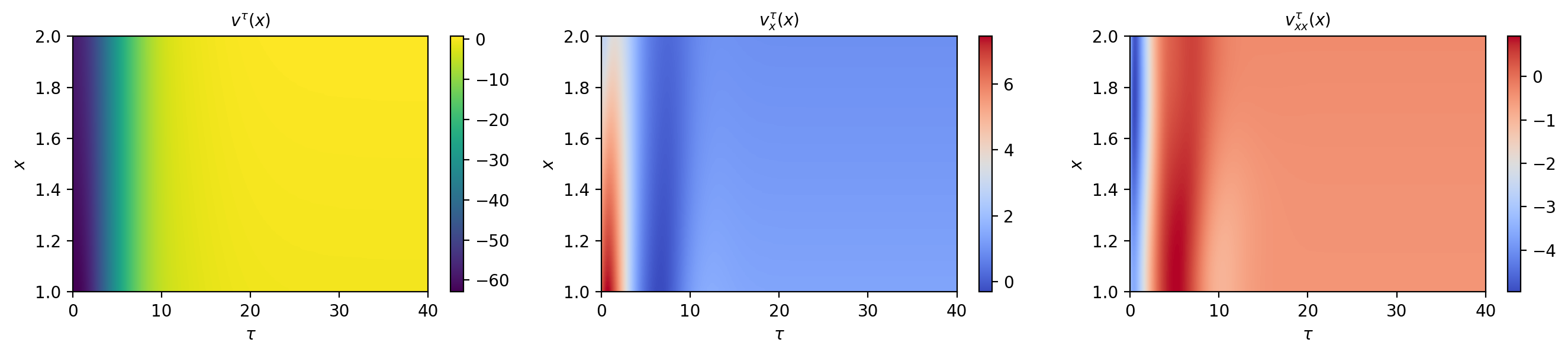}\\[4pt]
  \includegraphics[width=0.36\textwidth]{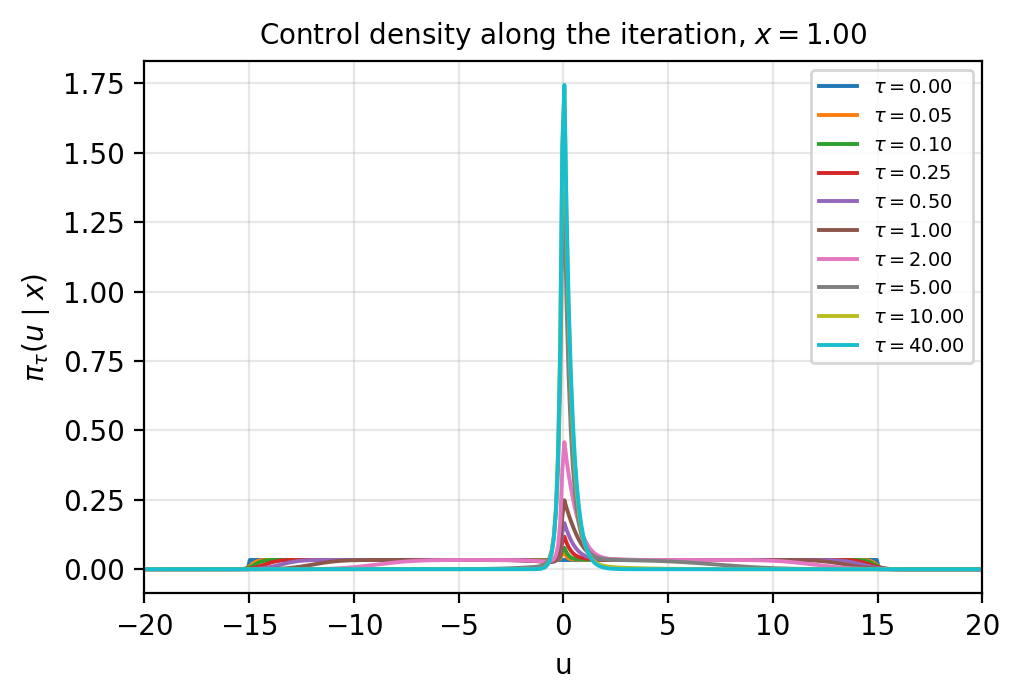}\hspace{0.03\textwidth}
  \includegraphics[width=0.36\textwidth]{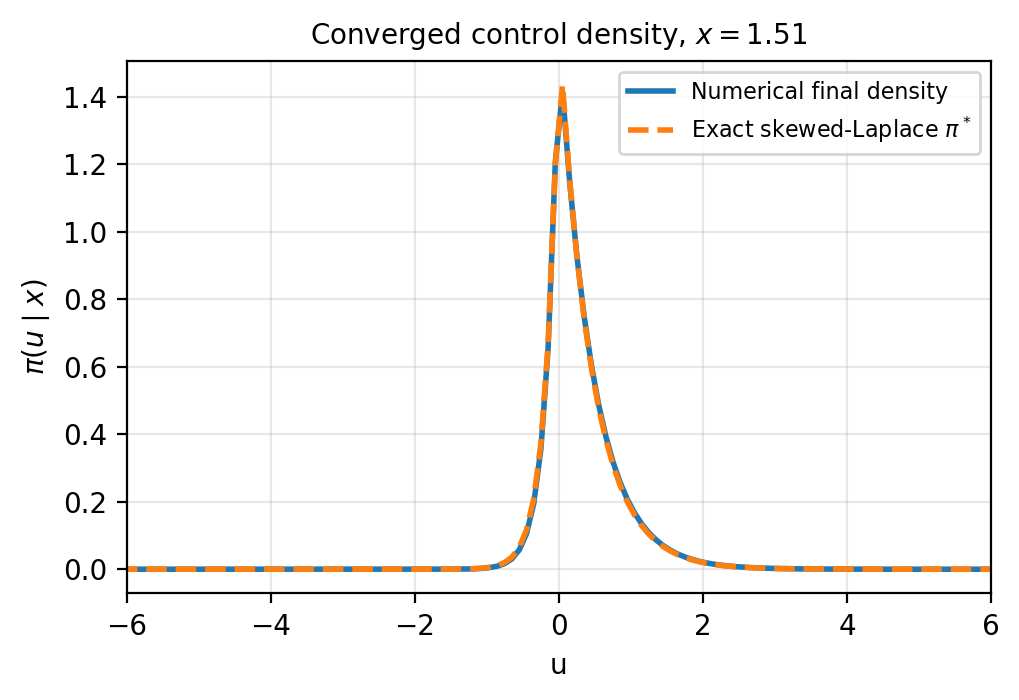}
  \caption{Top: heatmaps of the value iterates $v^\tau(x)$, $v^\tau_x(x)$ and $v^\tau_{xx}(x)$ over $(x,\tau)$; the three iterates relax from the Monte Carlo initialization and are stationary well before $\tau_{\mathrm{end}}=40$. Bottom left: snapshots of the control density $\pi_\tau(\cdot\,|\,x)$ at $x=1$ along the iteration, starting from the uniform density on $[-15,15]$. Bottom right: converged numerical density at $x=1.51$ versus the exact optimal density $\pi^*(\cdot\,|\,x)$; the two curves are visually indistinguishable.}
  \label{fig:skewlap}
\end{figure}

\bibliographystyle{siamplain}
\bibliography{references}

\begin{thebibliography}{10}

\bibitem{AmbrosioGigliSavare2008}
{\sc L.~Ambrosio, N.~Gigli, and G.~Savar\'e}, {\em Gradient Flows in Metric
  Spaces and in the Space of Probability Measures}, Lectures in Mathematics ETH
  Z\"urich, Birkh\"auser, Basel, 2nd~ed., 2008.

\bibitem{Aronson1967}
{\sc D.~G. Aronson}, {\em Bounds for the fundamental solution of a parabolic
  equation}, Bulletin of the American Mathematical Society, 73 (1967),
  pp.~890--896.

\bibitem{BS1991}
{\sc G.~Barles and P.~E. Souganidis}, {\em Convergence of approximation schemes
  for fully nonlinear second order equations}, Asymptotic analysis, 4 (1991),
  pp.~271--283.

\bibitem{bayraktar2024exponential}
{\sc E.~Bayraktar, Q.~Feng, and W.~Li}, {\em Exponential entropy dissipation
  for weakly self-consistent vlasov--vokker--planck equations}, Journal of
  Nonlinear science, 34 (2024).

\bibitem{bellman1954theory}
{\sc R.~Bellman}, {\em The theory of dynamic programming}, Bulletin of the
  American Mathematical Society, 60 (1954), pp.~503--515.

\bibitem{bellman1966dynamic}
{\sc R.~Bellman}, {\em Dynamic programming}, science, 153 (1966), pp.~34--37.

\bibitem{bertsekas2012dynamic}
{\sc D.~Bertsekas}, {\em Dynamic programming and optimal control: Volume I},
  vol.~4, Athena scientific, 2012.

\bibitem{bertsekas2015value}
{\sc D.~P. Bertsekas}, {\em Value and policy iterations in optimal control and
  adaptive dynamic programming}, IEEE transactions on neural networks and
  learning systems, 28 (2015), pp.~500--509.

\bibitem{bertsekas2011dynamic}
{\sc D.~P. Bertsekas et~al.}, {\em Dynamic programming and optimal control 3rd
  edition, volume ii}, Belmont, MA: Athena Scientific, 1 (2011).

\bibitem{BogachevKRS2015}
{\sc V.~I. Bogachev, N.~V. Krylov, M.~R\"ockner, and S.~V. Shaposhnikov}, {\em
  Fokker--Planck--Kolmogorov Equations}, vol.~207 of Mathematical Surveys and
  Monographs, American Mathematical Society, Providence, RI, 2015.

\bibitem{boltyanski1960maximum}
{\sc V.~Boltyanski, R.~Gamkrelidze, E.~Mishchenko, and L.~Pontryagin}, {\em The
  maximum principle in the theory of optimal processes of control}, IFAC
  Proceedings Volumes, 1 (1960), pp.~464--469.

\bibitem{chau2024continuous}
{\sc H.~Chau, D.~Nguyen, and T.~Nguyen}, {\em Continuous-time optimal
  investment with portfolio constraints: a reinforcement learning approach},
  arXiv preprint arXiv:2412.10692,  (2024).

\bibitem{chen2022uniform}
{\sc F.~Chen, Z.~Ren, and S.~Wang}, {\em Uniform-in-time propagation of chaos
  for mean field langevin dynamics}, arXiv preprint arXiv:2212.03050,  (2022).

\bibitem{CHS1987}
{\sc T.-S. Chiang, C.-R. Hwang, and S.~J. Sheu}, {\em Diffusion for global
  optimization in {R}\^{}n}, SIAM Journal on Control and Optimization, 25
  (1987), pp.~737--753.

\bibitem{chizat2022mean}
{\sc L.~Chizat}, {\em Mean-field langevin dynamics: Exponential convergence and
  annealing}, arXiv preprint arXiv:2202.01009,  (2022).

\bibitem{dai2023learning}
{\sc M.~Dai, Y.~Dong, and Y.~Jia}, {\em Learning equilibrium mean-variance
  strategy}, Mathematical Finance, 33 (2023), pp.~1166--1212.

\bibitem{Dalalyan2020}
{\sc A.~S. Dalalyan and L.~Riou-Durand}, {\em {On sampling from a log-concave
  density using kinetic Langevin diffusions}}, Bernoulli, 26 (2020), pp.~1956
  -- 1988.

\bibitem{denkert2025control}
{\sc R.~Denkert, H.~Pham, and X.~Warin}, {\em Control randomisation approach
  for policy gradient and application to reinforcement learning in optimal
  switching}, Applied Mathematics \& Optimization, 91 (2025), pp.~1--33.

\bibitem{dianetti2024exploratory}
{\sc J.~Dianetti, G.~Ferrari, and R.~Xu}, {\em Exploratory optimal stopping: A
  singular control formulation}, arXiv preprint arXiv:2408.09335,  (2024).

\bibitem{EGZ19}
{\sc A.~Eberle, A.~Guillin, and R.~Zimmer}, {\em Couplings and quantitative
  contraction rates for langevin dynamics}, Ann. Probab., 47 (2019),
  pp.~1982--2010.

\bibitem{FabesStroock1986}
{\sc E.~B. Fabes and D.~W. Stroock}, {\em A new proof of moser's parabolic
  harnack inequality using the old ideas of nash}, Archive for Rational
  Mechanics and Analysis, 96 (1986), pp.~327--338.

\bibitem{FN1984}
{\sc W.~H. Fleming and M.~Nisio}, {\em On stochastic relaxed control for
  partially observed diffusions}, Nagoya Mathematical Journal, 93 (1984),
  pp.~71--108.

\bibitem{friedman1964parabolic}
{\sc A.~Friedman}, {\em Partial Differential Equations of Parabolic Type},
  Prentice-Hall, Englewood Cliffs, NJ, 1964.

\bibitem{GW2020}
{\sc P.~Guasoni and G.~Wang}, {\em Consumption in incomplete markets}, Finance
  and Stochastics, 24 (2020), pp.~383--422.

\bibitem{HJE2018}
{\sc J.~Han, A.~Jentzen, and W.~E}, {\em Solving high-dimensional partial
  differential equations using deep learning}, Proceedings of the National
  Academy of Sciences, 115 (2018), pp.~8505--8510.

\bibitem{hancheolsen2010kolmogorov}
{\sc H.~Hanche-Olsen and H.~Holden}, {\em The {K}olmogorov--{R}iesz compactness
  theorem}, Expositiones Mathematicae, 28 (2010), pp.~385--394.

\bibitem{hu2021mean}
{\sc K.~Hu, Z.~Ren, D.~{\v{S}}i{\v{s}}ka, and {\L}.~Szpruch}, {\em Mean-field
  langevin dynamics and energy landscape of neural networks}, Annales de
  l'Institut Henri Poincare (B) Probabilites et statistiques, 57 (2021),
  pp.~2043--2065.

\bibitem{huang2022convergence}
{\sc Y.~Huang, Z.~Wang, and Z.~Zhou}, {\em Convergence of policy improvement
  for entropy-regularized stochastic control problems}, arXiv preprint
  arXiv:2209.07059,  (2022).

\bibitem{jacka2017policy}
{\sc S.~D. Jacka and A.~Mijatovi{\'c}}, {\em On the policy improvement
  algorithm in continuous time}, Stochastics, 89 (2017), pp.~348--359.

\bibitem{jia2022policy}
{\sc Y.~Jia and X.~Y. Zhou}, {\em Policy gradient and actor-critic learning in
  continuous time and space: Theory and algorithms}, Journal of Machine
  Learning Research, 23 (2022), pp.~1--50.

\bibitem{karatzas1991brownian}
{\sc I.~Karatzas and S.~E. Shreve}, {\em Brownian Motion and Stochastic
  Calculus}, vol.~113 of Graduate Texts in Mathematics, Springer-Verlag, New
  York, 2~ed., 1991.

\bibitem{kerimkulov2020exponential}
{\sc B.~Kerimkulov, D.~Siska, and L.~Szpruch}, {\em Exponential convergence and
  stability of howard's policy improvement algorithm for controlled
  diffusions}, SIAM Journal on Control and Optimization, 58 (2020),
  pp.~1314--1340.

\bibitem{krylov1996lectures}
{\sc N.~V. Krylov}, {\em Lectures on Elliptic and Parabolic Equations in
  H\"older Spaces}, vol.~12 of Graduate Studies in Mathematics, American
  Mathematical Society, Providence, RI, 1996.

\bibitem{lunardi1995analytic}
{\sc A.~Lunardi}, {\em Analytic Semigroups and Optimal Regularity in Parabolic
  Problems}, vol.~16 of Progress in Nonlinear Differential Equations and their
  Applications, Birkh\"auser, Basel, 1995.

\bibitem{ma2024convergence2}
{\sc J.~Ma, G.~Wang, and J.~Zhang}, {\em On convergence analysis of policy
  iteration algorithms for entropy-regularized stochastic control problems},
  arXiv preprint arXiv:2406.10959,  (2024).

\bibitem{ma2024convergence}
{\sc J.~Ma, G.~Wang, and J.~Zhang}, {\em On convergence and rate of convergence
  of policy improvement algorithms}, arXiv e-prints,  (2024), pp.~arXiv--2406.

\bibitem{possamai2024policy}
{\sc D.~Possama{\"\i} and L.~Tangpi}, {\em A policy iteration algorithm for
  non-markovian control problems}, arXiv preprint arXiv:2409.04037,  (2024).

\bibitem{puterman1981convergence}
{\sc M.~Puterman}, {\em On the convergence of policy iteration for controlled
  diffusions}, J. of Optimization Theory and Applications, 33 (1981),
  pp.~137--144.

\bibitem{puterman1979convergence}
{\sc M.~L. Puterman and S.~L. Brumelle}, {\em On the convergence of policy
  iteration in stationary dynamic programming}, Mathematics of Operations
  Research, 4 (1979), pp.~60--69.

\bibitem{raginsky2017non}
{\sc M.~Raginsky, A.~Rakhlin, and M.~Telgarsky}, {\em Non-convex learning via
  stochastic gradient langevin dynamics: a nonasymptotic analysis}, in
  Conference on Learning Theory, PMLR, 2017, pp.~1674--1703.

\bibitem{roberts1996exponential}
{\sc G.~O. Roberts and R.~L. Tweedie}, {\em Exponential convergence of langevin
  distributions and their discrete approximations}, Bernoulli, 2 (1996),
  pp.~341--363.

\bibitem{Sheu1991}
{\sc S.-J. Sheu}, {\em Some estimates of the transition density of a
  nondegenerate diffusion markov process}, The Annals of Probability, 19
  (1991), pp.~538--561.

\bibitem{vsivska2024gradient}
{\sc D.~{\v{S}}i{\v{s}}ka and {\L}.~Szpruch}, {\em Gradient flows for
  regularized stochastic control problems}, SIAM Journal on Control and
  Optimization, 62 (2024), pp.~2036--2070.

\bibitem{szepesvari2022algorithms}
{\sc C.~Szepesv{\'a}ri}, {\em Algorithms for reinforcement learning}, Springer
  nature, 2022.

\bibitem{TZZ2022}
{\sc W.~Tang, Y.~P. Zhang, and X.~Y. Zhou}, {\em Exploratory hjb equations and
  their convergence}, SIAM Journal on Control and Optimization, 60 (2022),
  pp.~3191--3216.

\bibitem{tang2024regret}
{\sc W.~Tang and X.~Y. Zhou}, {\em Regret of exploratory policy improvement and
  $ q $-learning}, arXiv preprint arXiv:2411.01302,  (2024).

\bibitem{tran2024policy}
{\sc H.~V. Tran, Z.~Wang, and Y.~P. Zhang}, {\em Policy iteration for
  exploratory hamilton--jacobi--bellman equations}, arXiv preprint
  arXiv:2406.00612,  (2024).

\bibitem{Villani2009}
{\sc C.~Villani}, {\em Optimal Transport: Old and New}, vol.~338 of Grundlehren
  der mathematischen Wissenschaften, Springer, Berlin, 2009.

\bibitem{wang2020reinforcement}
{\sc H.~Wang, T.~Zariphopoulou, and X.~Y. Zhou}, {\em Reinforcement learning in
  continuous time and space: A stochastic control approach}, Journal of Machine
  Learning Research, 21 (2020), pp.~1--34.

\bibitem{wang2020continuous}
{\sc H.~Wang and X.~Y. Zhou}, {\em Continuous-time mean--variance portfolio
  selection: A reinforcement learning framework}, Mathematical Finance, 30
  (2020), pp.~1273--1308.

\bibitem{welling2011bayesian}
{\sc M.~Welling and Y.~W. Teh}, {\em Bayesian learning via stochastic gradient
  langevin dynamics}, in Proceedings of the 28th international conference on
  machine learning (ICML-11), Citeseer, 2011, pp.~681--688.

\bibitem{zhang2017hitting}
{\sc Y.~Zhang, P.~Liang, and M.~Charikar}, {\em A hitting time analysis of
  stochastic gradient langevin dynamics}, in Conference on Learning Theory,
  PMLR, 2017, pp.~1980--2022.

\bibitem{zhao2023policy}
{\sc H.~Zhao, W.~Tang, and D.~Yao}, {\em Policy optimization for continuous
  reinforcement learning}, Advances in Neural Information Processing Systems,
  36 (2023), pp.~13637--13663.

\bibitem{zhou2023policy}
{\sc M.~Zhou and J.~Lu}, {\em A policy gradient framework for stochastic
  optimal control problems with global convergence guarantee}, arXiv preprint
  arXiv:2302.05816,  (2023).

\bibitem{zhou2024solving}
{\sc M.~Zhou and J.~Lu}, {\em Solving time-continuous stochastic optimal
  control problems: Algorithm design and convergence analysis of actor-critic
  flow}, arXiv preprint arXiv:2402.17208,  (2024).

\bibitem{ZZH2025}
{\sc M.~Zhou, H.~Zhou, and R.~Hu}, {\em Learning mean-field games through
  mean-field actor-critic flow}, arXiv preprint arXiv:2510.12180,  (2025).

\bibitem{ZHou1992}
{\sc X.~Y. Zhou}, {\em On the existence of optimal relaxed controls of
  stochastic partial differential equations}, SIAM journal on control and
  optimization, 30 (1992), pp.~247--261.

\end{thebibliography}

\end{document}